\documentclass[12pt]{amsart}

\usepackage[utf8]{inputenc}
\usepackage[USenglish]{babel}
\usepackage{amsmath,amsthm,amssymb,amsfonts}
\usepackage{mathtools}
\usepackage{mathrsfs}
\usepackage{bbm}

\usepackage{indentfirst}
\usepackage{enumerate}
\usepackage{accents}
\usepackage{enumitem}

\usepackage{float}
\usepackage{graphicx}
\usepackage{wrapfig}
\usepackage{placeins}
\usepackage{caption}
\usepackage{arydshln}

\usepackage{tikz}
\usetikzlibrary{patterns}
\usepackage{tikz-cd}
\usepackage{tikz-3dplot}

\usetikzlibrary{decorations.pathreplacing,decorations.markings}
\usetikzlibrary{babel}

\tikzset{                        
    symbol/.style={
        ,draw=none
        ,every to/.append style={%
            edge node={node [sloped, allow upside down, auto=false]{$#1$}}}
    }
}

\usepackage{chemgreek,textgreek} 

\usepackage[breaklinks=true, bookmarksopenlevel=1, bookmarksdepth=2]{hyperref}

\allowdisplaybreaks

\newtheorem{thm}{Theorem}[section]
\newtheorem{cor}[thm]{Corollary}
\newtheorem{lem}[thm]{Lemma}
\newtheorem{prop}[thm]{Proposition}

\theoremstyle{plain} 
\newcommand{\thistheoremname}{}
\newtheorem*{genericthm}{\thistheoremname}

\theoremstyle{definition}

\theoremstyle{remark}
\newtheorem{rem}[thm]{Remark}
\newtheorem*{xrem}{Remark}

\numberwithin{equation}{section}

\newcommand{\Z}{\mathbb{Z}}      
\newcommand{\Q}{\mathbb{Q}}      
\newcommand{\R}{\mathbb{R}}      

\newcommand{\norm}[1]{\left\lVert #1\right\rVert} 

\newcommand\restr[2]{{           
  \left.\kern-\nulldelimiterspace #1%
  \right|_{#2}%
 }}

\newcommand\minus{               
  \setbox0=\hbox{-}%
  \vcenter{%
    \hrule width\wd0 height \the\fontdimen8\textfont3%
  }%
}

\renewcommand{\pmod}[1]{         
  ~(\mathrm{mod}~#1)}

\newcommand{\lmod}[3]{         
  #1 \equiv #2 \pmod{#3}}

\newcommand{\Fr}{\mathrm{Fr}}   
\renewcommand{\P}{\mathcal{P}}

\newcommand{\ul}[1]{\underline{#1}}
\newcommand{\ub}[1]{\mathbf{#1}}
\newcommand{\cH}{\mathcal{H}} 
\newcommand{\ex}{\mathrm{ex}}

\restylefloat{table}

\setlist[itemize]{leftmargin=*}
\setlist[enumerate]{leftmargin=*}

\makeatletter
 \newcommand*\alphgreek[1]{\expandafter\@alphgreek\csname c@#1\endcsname}
 \newcommand*\@alphgreek[1]{\csname chemgreek_int_to_greek:n\endcsname{#1}}
 \newcommand*\Alphgreek[1]{\expandafter\@Alphgreek\csname c@#1\endcsname}
 \newcommand*\@Alphgreek[1]{\csname chemgreek_int_to_Greek:n\endcsname{#1}}

 \AddEnumerateCounter*{\alphgreek}{\@alphgreek}{\chemalpha}
 \AddEnumerateCounter*{\Alphgreek}{\@Alphgreek}{\chemAlpha}
\makeatother

\makeatletter
\def\@tocline#1#2#3#4#5#6#7{\relax
  \ifnum #1>\c@tocdepth 
  \else
    \par \addpenalty\@secpenalty\addvspace{#2}%
    \begingroup \hyphenpenalty\@M
    \@ifempty{#4}{%
      \@tempdima\csname r@tocindent\number#1\endcsname\relax
    }{%
      \@tempdima#4\relax
    }%
    \parindent\z@ \leftskip#3\relax \advance\leftskip\@tempdima\relax
    \rightskip\@pnumwidth plus4em \parfillskip-\@pnumwidth
    #5\leavevmode\hskip-\@tempdima
      \ifcase #1
       \or\or \hskip 1em \or \hskip 2em \else \hskip 3em \fi%
      #6\nobreak\relax
    \dotfill\hbox to\@pnumwidth{\@tocpagenum{#7}}\par
    \nobreak
    \endgroup
  \fi}
\makeatother

\begin{document}

\title{On the size and structure of \texorpdfstring{$\MakeLowercase{t}$}{\MakeLowercase{t}}-representable sumsets}%
\author{Christian T\'afula}%
\address{D\'epartment de Math\'ematiques et Statistique, %
 Universit\'e de Montr\'eal, %
 CP 6128 succ Centre-Ville, %
 Montreal, QC H3C 3J7, Canada}%
\email{christian.tafula.santos@umontreal.ca}%


 \begin{abstract}
  Let $A\subseteq \mathbb{Z}_{\geq 0}$ be a finite set with minimum element $0$, maximum element $m$, and $\ell$ elements strictly in between. Write $(hA)^{(t)}$ for the set of integers that can be written in at least $t$ ways as a sum of $h$ elements of $A$. We prove that $(hA)^{(t)}$ is ``structured'' for
  \[ h \geq (1+o(1)) \frac{1}{e} m\ell t^{1/\ell} \]
  (as $\ell \to \infty$, $t^{1/\ell} \to \infty$), and prove a similar theorem on the size and structure of $A\subseteq \mathbb{Z}^d$ for $h$ sufficiently large. Moreover, we construct a family of sets $A = A(m,\ell,t)\subseteq \mathbb{Z}_{\geq 0}$ for which $(hA)^{(t)}$ is not structured for $h\ll m\ell t^{1/\ell}$.
 \end{abstract}
\maketitle

\section{Introduction}
 For an integer $m\geq 0$, write $[m]:= \{0,1,\ldots,m\}$. Given a finite set of integers $A$, we are interested in determining the elements of the sumset 
 \[ hA = \bigg\{ \sum_{a\in A} k_a\cdot a ~\bigg|~ k_a\in \Z_{\geq 0},\, \sum_{a\in A} k_a = h \bigg\}. \]
 Throughout we will suppose, without loss of generality, that $\min_{a\in A} a = 0$ and $\gcd(A) = 1$,\footnote{For if $A' := \{ (a-a_0)/d ~|~ a\in A\}$, where $a_0 := \min_{a\in A} a$, $d:= \gcd\{a-a_0 ~|~ a\in A\}$, then $hA$ can be recovered from $hA'$ after an affine transformation: $hA = ha_0 + d\cdot (hA')$, where $d\cdot B = \{d\cdot a ~|~ a\in B\}$.} and write 
 \[ A=\{0 = a_0 < a_1 <\cdots < a_{\ell} < a_{\ell+1} = m\} \]
 (so $|A| = \ell+2$), where $m$ denotes the maximum element of $A$.
 
 Writing
 \[ \P(A) := \bigcup_{h\geq 1} hA \]
 for the set of all integers expressible as a finite sum of elements of $A$, define the \emph{exceptional set} of $A$ as
 \[ \mathcal{E}(A) := \Z_{\geq 0}\setminus \P(A). \]
 So, since $A\subseteq 2A \subseteq 3A \subseteq \cdots$, we have $hA \subseteq [hm]\setminus \mathcal{E}(A)$. Moreover, writing $m-A = \{m-a ~|~ a\in A\}$, we have $h(m-A) = hm- hA$, so if $n\in \mathcal{E}(m-A)$ then $hm-n\notin hA$. Therefore,
 \begin{equation}
  hA \subseteq [hm]\setminus \big(\mathcal{E}(A) \cup (hm-\mathcal{E}(m-A)) \big). \label{str1}
 \end{equation}
 If equality holds true in \eqref{str1}, then $hA$ is said to be \emph{structured}.\footnote{One might also say that $hA$ \emph{stabilizes} for large $h$.} Nathanson \cite{natSOFSI} showed in 1972 the existence of a smallest $h_1= h_1(A)$ such that $hA$ is structured for every $h\geq h_1$. Nathanson himself proved that $h_1 \leq m^2(\ell+1)$; this was improved to $h_1 \leq \sum_{i=2}^{\ell+1} (a_i - 1) -1$ by Chen--Chen--Wu \cite{chechewu11}, then $h_1\leq 2\lfloor m/2\rfloor$ by Granville--Shakan \cite{grasha20}, and then $h_1\leq m-\ell$ by Granville--Walker \cite{grawal21}. In 2022, Lev \cite{lev22} showed that if $h\leq \max\{m - \tfrac{3}{2}(\ell-1),\, \tfrac{2}{3}(m-\ell)\}$ then $hA$ is structured except if $A$ or $m-A = \{0,1,h+2,\ldots,m\}$, in which case Granville--Walker's ``$m-\ell$'' is sharp.
 
 \subsection{\texorpdfstring{$t$}{t}-representables}\label{ssec12}
 What if instead of integers that can be represented at least once as a sum of $h$ elements of $A$, we asked the same question for integers that can be represented at least two times? Or $t$ times? Define
 \begin{align*}
  r_{A,h}(n) :=&\, \#\{(k_0,\ldots,k_{\ell+1})\in \Z_{\geq 0}^{\ell+2} ~|~ k_0 a_0 +\cdots + k_{\ell+1}a_{\ell+1} = n,\ \textstyle\sum_{i=0}^{\ell+1} k_i = h\} \\
  =&\, \#\{(k_1,\ldots,k_{\ell+1})\in \Z_{\geq 0}^{\ell+1} ~|~ k_1a_1 +\cdots + k_{\ell+1}a_{\ell+1} = n,\ \textstyle\sum_{i=1}^{\ell+1} k_i \leq h\}
 \end{align*}
 (since $a_0=0$) for the \emph{representation function} of $hA$. Note that $a_1<a_2<\cdots < a_{\ell+1}$, so switching the order of summation does not get counted as a new sum. For every integer $t\geq 1$, define
 \[ (hA)^{(t)} := \{n\in \Z_{\geq 0} ~|~ r_{A,h}(n) \geq t\}. \]
 In 2021, Nathanson \cite{natSOFSII} showed that $(hA)^{(t)}$ enjoys a similar notion of structure for large $h$. Writing $\P_t(A) := \bigcup_{h\geq 1} (hA)^{(t)}$, define the \emph{$t$-exceptional set} of $A$ as $\mathcal{E}_t(A) := \Z_{\geq 0} \setminus \P_t(A)$. Then, there exists a smallest $h_t = h_t(A)$ such that
 \begin{equation}
  (hA)^{(t)} = [hm]\setminus \big(\mathcal{E}_t(A) \cup (hm-\mathcal{E}_t(m-A)) \big) \label{strt}
 \end{equation}
 for every $h \geq h_t$. When \eqref{strt} holds, $(hA)^{(t)}$ is said to be \emph{structured}. Nathanson showed that $h_t\leq m\ell\, (tm-1) +1$, which was improved to $h_t \leq \sum_{i=2}^{\ell+1} (ta_i - 1) -1$ by Yang--Zhou \cite{yanzho21} in the same year.
 
 The \emph{Frobenius number}
 \[ \Fr(A) := \max_{n\in\mathcal{E}(A)} n \]
 \textnormal{(with $\Fr(A) := 0$ if $\mathcal{E}(A)=\varnothing$)} is easily seen to be finite,\footnote{As is well-known, there are $x_1,\ldots, x_{\ell+1}\in \Z$ such that $\sum_{i=1}^{\ell+1} a_i x_i = 1$. Taking $M := a_1\max_{i} |x_i|$ and $N := \sum_{i=1}^{\ell+1} Ma_i$, we have $N+k = \sum_{i=1}^{\ell+1} (M+k x_i)a_i \in \P(A)$ for every $0\leq k< a_1$. Therefore, since $a_1 + \P(A) \subseteq \P(A)$, we have $\{N,N+1,\ldots\} \in \P(A)$, so $\Fr(A) < N$.} and the problem of estimating $\Fr(A)$ dates back at least to Sylvester {\cite[p. 134]{syl1882}}, who in 1882 showed that if $A = \{0<a_1<a_2\}$, then $\Fr(A) = a_1a_2 - a_1 - a_2$. Analogously, define
 \[ \Fr_t(A) := \max_{n\in\mathcal{E}_t(A)} n, \]
 which will be shown to be finite in Corollary \ref{lbrhoC}.
 
 \begin{xrem}[Case $\ell=1$]
  The case $|A|=3$, which is discussed in the appendix at the end, enjoys the property that $(hA)^{(t)}$ is structured for every $h$, $t\geq 1$. The case $t=1$ was proven by Granville--Shakan \cite{grasha20}. Accordingly, in Theorems \ref{MT1}, \ref{MT2} we only consider $\ell \geq 2$.
 \end{xrem}
 
 \begin{thm}\label{MT1}
  Suppose $|A|\geq 4$. The set $(hA)^{(t)}$ is structured as in \eqref{strt} for every
  \[ h \geq \left\lfloor\frac{\Fr_t(A)+m}{a_1}\right\rfloor + \left\lfloor\frac{\Fr_t(m-A)+m}{m-a_{\ell}}\right\rfloor. \]
 \end{thm}
 
 This bound can be simplified to:
 
 \begin{thm}\label{MT2}
  Suppose $|A|\geq 4$. The set $(hA)^{(t)}$ is structured as in \eqref{strt} for every
  \begin{equation*}
   h\geq C_{A,t}\, \frac{1}{e}m \ell t^{1/\ell},
  \end{equation*}
  where
  \begin{equation*}
   C_{A,t} \leq \bigg(1+\frac{4}{\ell}\bigg)\frac{e}{t^{1/\ell}} + \bigg(1+\frac{2}{\ell}\bigg)\frac{1 + (\log 4\ell)/\ell}{\min\{a_1,m-a_{\ell}\}}.
  \end{equation*}
  In particular, $C_{A,t} \leq 3e$ if $\ell\geq 4$, and $C_{A,t} \leq 1 + o(1)$ as $\ell \to \infty$, $t^{1/\ell}\to\infty$.
 \end{thm}
 
 This improves Yang--Zhou's bound for, e.g., $t\geq 8\ell \geq 32$ -- see Remark \ref{compYZ}.
 
 In Section \ref{sec4}, we construct a family of finite sets $A = A(m,\ell,t)$ for certain $m\geq 5$, $\ell\geq 2$, $t\geq 2$, each with $1$, $m-1\in A$, for which
 \[ h_t(A) \geq (1 + o(1))\, \frac{1}{e} m\ell t^{1/\ell} \]
 as $\ell\to \infty$, $t^{1/\ell}\to \infty$. So Theorem \ref{MT2} is asymptotically tight in this regime.
 
 \subsection{Structure in \texorpdfstring{$\Z^d$}{Z\^{}d}}\label{ssec13}
 A version of Theorem \ref{MT1} holds for finite sets $A$ of arbitrary dimension. Let $d\geq 1$ be a fixed integer and $A\subseteq \Z^d$ be a finite set. Suppose without loss of generality that $\mathrm{span}_{\R}(A) = \big\{\sum_{\ul{a}\in A} c_{\ul{a}}\ul{a} ~|~ c_{\ul{a}}\in \R\big\} = \R^{d}$.\footnote{If $\mathrm{span}_{\R}(A) \simeq \R^{d'}$ for some $d'<d$, then after a change of basis, suppose $A\subseteq \Z^{d'} := \R^{d'} \cap \Z^{d}$.} The \emph{convex hull} of $A$ is given by 
 \[ \cH(A) := \bigg\{\sum_{\ul{a}\in A} c_{\ul{a}} \ul{a} ~\bigg|~ c_{\ul{a}}\in \R_{\geq 0},\, \sum_{\ul{a} \in A} c_{\ul{a}} = 1 \bigg\} \]
 and its set of \emph{extremal points} $\ex(\cH(A))$ (or ``corners'' of the boundary of $\mathcal{H}(A)$) as the set of $\ul{v} \in \cH(A)$ for which there exists a hyperplane in $\R^d$ tangent to $\cH(A)$ at $\ul{v}$, such that $\cH(A)\setminus\{\ul{v}\}$ lies entirely on one side of the hyperplane. Explicitly,
 \[ \ex(\cH(A)) := \bigg\{\ul{v}\in \cH(A) ~\bigg|~ \begin{array}{@{}c@{}} \exists \ul{\mathfrak{n}} = \ul{\mathfrak{n}}(\ul{v})\in \Z^{d},\, \exists c\in\R \text{ such that}\\ \langle \ul{\mathfrak{n}},\ul{v}\rangle = c \text{ and } \langle \ul{\mathfrak{n}},\ul{x}\rangle > c,\,\forall \ul{x}\in \cH(A)\setminus\{\ul{v}\} \end{array}\bigg\}. \]
 
 \FloatBarrier
 \tdplotsetmaincoords{110}{-20}
 \begin{figure}[!htb]
 \centering
 \begin{tikzpicture}[tdplot_main_coords, scale=0.8, every node/.style={scale=0.8}]
  \foreach \x in {0,...,5}
   \foreach \y in {0,...,5}
    \draw[black!20] (\x,\y,0) -- (\x,\y,5) (0,\x,\y) -- (5,\x,\y) (\x,0,\y) -- (\x,5,\y);
  
  \draw[thick,->] (0,0,0) -- (5,0,0) node[anchor=north east]{$x$};
  \draw[thick,->] (0,0,0) -- (0,5,0) node[anchor=north west]{$y$};
  \draw[thick,->] (0,0,0) -- (0,0,5) node[anchor=south]{$z$};
  
  \path (0,4,4) coordinate (A) 
  (4,4,2) coordinate (B) 
  (5,0,5) coordinate (C) 
  (2,0,5) coordinate (D) 
  (0,0,0) coordinate (O);
  
  \draw[green!50!black, fill=green!50!black!50, opacity=.8] (O) -- (A) -- (D) -- cycle;
  \draw[green!50!black, fill=green!50!black!50, opacity=.8] (O) -- (D) -- (C) -- cycle;
  \draw[green!50!black, fill=green!50!black!50, opacity=.8] (O) -- (C) -- (B) -- cycle;
  
  \path (2,1,3) coordinate (P)
  (3,3,2) coordinate (Q);
  \draw[-Stealth,dashed,thick,color=green!50!black] (0,0,0) -- (P);
  \node[tdplot_main_coords, color=green!50!black] at (P){$\circ$};
  \draw[-Stealth,dashed,thick,color=green!50!black] (0,0,0) -- (Q);
  \node[tdplot_main_coords, color=green!50!black] at (Q){$\circ$};
  
  \draw[green!50!black, fill=green!50!black!60, opacity=.5] (O) -- (B) -- (A) -- cycle;

  \draw[-Stealth,thick,color=green!50!black] (0,0,0) -- (A);
  \node[tdplot_main_coords, color=green!50!black] at (A){$\bullet$};
  \draw[-Stealth,thick,color=green!50!black] (0,0,0) -- (B);
  \node[tdplot_main_coords, color=green!50!black] at (B){$\bullet$};
  \draw[-Stealth,thick,dashed,color=green!50!black] (0,0,0) -- (C);
  \node[tdplot_main_coords, color=green!50!black] at (C){$\bullet$};
  \draw[-Stealth,thick,dashed,color=green!50!black] (0,0,0) -- (D);
  \node[tdplot_main_coords, color=green!50!black] at (D){$\bullet$};
 \end{tikzpicture}
 \caption{The set $A$ $=$ $\{\mathbf{(0,0,0)}$, $\mathbf{(0,4,4)}$, $\mathbf{(2,0,5)}$, $(2,1,3)$, $(3,3,2)$, $\mathbf{(4,4,2)}$, $\mathbf{(5,0,5)}\}$. Points written in bold compose $\mathrm{ex}(\cH(A))$.}
 \end{figure} 
 
 Analogous to the situation in $\Z$, suppose without loss of generality $\ul{0}\in \ex(\cH(A))$, and for every integer $h \geq 1$, define the \emph{representation function} of $hA$ as
 \begin{align*}
  r_{A,h}(\ul{p}) :=&\, \#\bigg\{(k_a)_{a\in A} \in \Z_{\geq 0}^{|A|} ~\bigg|~ \sum_{\ul{a}\in A} k_{\ul{a}} \ul{a} = \ul{p},\text{ with } \sum_{\ul{a}\in A} k_{\ul{a}} = h \bigg\} \\
  =&\, \#\bigg\{(k_a)_{a\in A\setminus\{0\}} \in \Z_{\geq 0}^{|A|-1} ~\bigg|~ \sum_{\ul{a}\in A\setminus\{\ul{0}\}} k_{\ul{a}} \ul{a} = \ul{p},\text{ with } \sum_{\ul{a}\in A\setminus\{\ul{0}\}} k_{\ul{a}} \leq h \bigg\},
 \end{align*}
 For every integer $t\geq 1$, define
 \[ (hA)^{(t)} := \{\ul{p} \in \Z^{d} ~|~ r_{A,h}(\ul{p}) \geq t\}. \]
 Writing $\mathcal{C}_A := \big\{\sum_{\ul{a}\in A} c_{\ul{a}}\ul{a} ~|~ c_{\ul{a}}\in \R_{\geq 0}\big\}$ for the \emph{cone} of $A$ (note that $\cH(A)\subseteq \mathcal{C}_A$), and $\Lambda_{A} := \mathrm{span}_{\Z}(A) = \big\{\sum_{\ul{a}\in A} c_{\ul{a}}\ul{a} ~|~ c_{\ul{a}}\in \Z\big\}$ for the $\Z$-span of $A$, we can define
 \[ \P_t(A) := \bigcup_{h\geq 1} (hA)^{(t)},\qquad \mathcal{E}_t(A) := (\mathcal{C}_A\cap \Lambda_A)\setminus \P_t(A). \]
 We will prove that there is a smallest $h_t = h_t(A)$ such that
 \begin{equation}
  (hA)^{(t)} = (h \cH(A) \cap \Lambda_{A}) \setminus \bigg( \bigcup_{\ul{v}\,\in\,\ex(\cH(A))} (h\ul{v} - \mathcal{E}_t(\ul{v}-A)) \bigg) \label{strzd}
 \end{equation}
 for every $h\geq h_t$, where $h\mathcal{H}(A)= \{h\ul{p} ~|~ \ul{p}\in \mathcal{H}(A)\}$ is the $h$-dilate of $\mathcal{H}(A)$. In 2020, Granville--Shakan \cite{grasha20} proved the finiteness of $h_1$, and in 2023 Granville--Shakan--Walker \cite{grashawal23} obtained the explicit bound
 \[ h_1(A) \leq (d+1)2^{11d^2} d^{12d^{6}}|A|^{3d^2} \Big(\max_{\ul{a},\ul{b}\in A}\norm{\ul{a}-\ul{b}}_{\infty}\Big)^{8d^{6}} \leq \Big(d\ell\,\max_{\ul{a},\ul{b}\in A}\norm{\ul{a}-\ul{b}}_{\infty}\Big)^{13d^6}. \]
 
 Our result gives a bound for $h_t$ depending on certain quantities associated to $A$. Let $\ul{\mathfrak{n}}$ be a unit vector in $\R^d$ for which $\langle \ul{\mathfrak{n}}, \ul{v}\rangle > 0$ for every $\ul{v} \in \mathcal{C}_A \setminus\{\ul{0}\}$, so that $\mathcal{C}_A \setminus\{\ul{0}\}$ lies entirely on one side of the hyperplane normal to $\ul{\mathfrak{n}}$. Define
 \begin{equation}
  \delta_{A} = \delta_{A,\ul{\mathfrak{n}}} := \min_{\ul{a}\in A\setminus\{\ul{0}\}} \langle\ul{a},\ul{\mathfrak{n}} \rangle,\qquad \Delta_{A} = \Delta_{A,\ul{\mathfrak{n}}} := \max_{\ul{a}\in A\setminus\{\ul{0}\}} \langle\ul{a},\ul{\mathfrak{n}} \rangle. \label{defd}
 \end{equation}
 That is, we project the elements of $A$ onto the $1$-dimensional cone $\R_{\geq 0}\,\ul{\mathfrak{n}}$, and take $\delta_{A}$ (resp., $\Delta_{A}$) to be the length of the \emph{smallest} (resp. \emph{largest}) \emph{non-zero projection}. These constants are used in the estimate of the structure theorem.

 \begin{lem}\label{mainlm} 
  There exists a minimum  $\varphi = \varphi_{A,t}\in \R_{\geq 1}$ such that, for every real $\lambda \geq \varphi$, we have
  \[ \big((\lambda\cH(A) \cap \Lambda_A)\setminus \mathcal{E}_t(A)\big) + A = \big((\lambda+1) \cH(A) \cap \Lambda_A\big)\setminus \mathcal{E}_t(A). \]
 \end{lem}
 
 \begin{thm}[$t$-structure]\label{ZdFrobt}
  The set $(hA)^{(t)}\subseteq \Z^d$ is structured as in \eqref{strzd} for every
  \[ h \geq \max_{\substack{\ul{0}\,\in\,B\, \subseteq\, \mathrm{ex}(\cH(A)) \\ |B|\,=\,d+1 \\ \mathrm{span}_{\R}(B) \,=\, \R^d}} \bigg(\sum_{\ul{b}\in B} \left\lceil\frac{\Delta_{\ul{b}-A}}{\delta_{\ul{b}-A}}\, \varphi_{\ul{b}-A,t}\right\rceil \bigg), \]
  where $\varphi_{\ul{b}-A,t}$ is the constant from Lemma \ref{mainlm} for the set $\ul{b}-A$.
 \end{thm}
 
 \begin{xrem}
  Note that the bound in Theorem \ref{ZdFrobt} depends on the choice of $\ul{\mathbf{n}}$ in \eqref{defd}. The bound is valid for all valid choices of $\ul{\mathfrak{n}}$, so one could, in principle, find the $\ul{\mathfrak{n}}$ that yields the minimum $\Delta_A/\delta_A$.
 \end{xrem}
 
 Lemma \ref{mainlm} and Theorem \ref{ZdFrobt} are proven in Section \ref{defdelta}. We note that Theorem \ref{ZdFrobt} can be seen as a generalization of Theorem \ref{MT1} for $d\geq 2$ -- see Remark \ref{remana}.

\subsection{Size of \texorpdfstring{$(hA)^{(t)}$}{(hA)\^{}(t)}}
 In Section \ref{defdelta}, we will also prove the following: 
 
 \begin{thm}[$t$-Khovanskii]\label{tKhov}
  If $A\subseteq \Z^{d}$ is finite, then for every $t\geq 1$ there is $h^{\mathrm{Kh}}_t(A) \in \Z_{\geq 0}$ such that, for every $h\geq h^{\mathrm{Kh}}_t(A)$, we have
  \[ |(hA)^{(t)}| = p_{A,t}(h), \] 
  where $p_{A,t}(x)\in \Q[x]$ is a polynomial of degree $\leq d$.
 \end{thm}

 The case $t=1$ is Khovanskii's theorem (cf. \cite{curgol21, grashawal23} for effective bounds to $h^{\mathrm{Kh}}_{1}(A)$).

\section{Estimating \texorpdfstring{$\Fr_t(A)$}{Frob\_t(A)}}\label{sec2}
 Let $m$, $\ell$, $t$, and $A$ be as in Subsection \ref{ssec12}. Write $r_A(n) := \lim_{h\to\infty} r_{A,h}(n)$ for the \emph{total representation function} of $A$, so that $\P_t(A) = \{n\in \Z_{\geq 0} ~|~ r_A(n) \geq t\}$.
 
 \begin{prop}\label{lbrho}
  For every $n\geq 0$, we have
  \[ r_A(n) \,\leq\, \frac{1}{\ell!} \frac{(n + a_1\sum_{j=2}^{\ell+1} a_{j})^{\ell}}{a_1 \cdots a_{\ell}m}. \]
  Moreover, if $n\geq (a_1-1)\sum_{j=2}^{\ell+1} a_j$, then
  \[ r_A(n) \,\geq\, \frac{1}{\ell!} \frac{(n-(a_1-1)\sum_{j=2}^{\ell+1} a_j)^{\ell}}{a_1\cdots a_{\ell}m}  \]
 \end{prop}
 
 To prove this, we will need the following standard lattice point counting lemma, to which we include a short proof for the sake of completeness.
 
 \begin{lem}\label{clp}
  Let $d\geq 2$ be an integer, and $N_1$, $N_2$, $\ldots$, $N_d \in \Z_{\geq 0}$. For $R\in \R_{\geq 0}$, define the $d$-simplex
  \[ \varDelta_R(N_1,\ldots, N_d) := \{(x_1,\ldots,x_{d}) \in (\R_{\geq 0})^{d} ~|~ x_1N_1 + x_2N_2 + \cdots + x_{d}N_{d} \leq R \}. \]
  Then,
  \begin{equation}
   \text{\small$\mathrm{vol}_{\R^d}\, \varDelta_R(N_1,\ldots, N_d) \leq \big|\varDelta_R(N_1,\ldots, N_d) \cap \Z^d\big| \leq \mathrm{vol}_{\R^d}\, \varDelta_{R +\sum_{i=1}^{d} N_i}(N_1,\ldots, N_d),$} \label{vol2s}
  \end{equation}
  and
  \begin{equation}
   \mathrm{vol}_{\R^d}\, \varDelta_R(N_1,\ldots, N_d) = \frac{1}{d!} \frac{R^d}{N_1\cdots N_d}. \label{Rd}
  \end{equation}
 \end{lem}
 \begin{proof}
  To each point $\underline{p} = (p_1,\ldots, p_d) \in \varDelta_R(N_1,\ldots, N_d) \cap \Z^d$ consider the unit hypercube
  \[ H(\underline{p}) := \{(p_1 + y_1, \ldots, p_d + y_d) ~|~ 0\leq y_1,\ldots,y_d \leq 1 \} = \underline{p} + H(\ul{0}). \]
  The union $\bigcup_{\underline{p}\in \varDelta_R(N_1,\ldots, N_d) \,\cap\, \Z^d} H(\underline{p})$ covers $\varDelta_R(N_1,\ldots, N_d)$. Therefore
  \begin{align*}
   |\varDelta_R(N_1,\ldots, N_d) \,\cap\, \Z^d| &= \mathrm{vol}_{\R^{d}} \bigg(\bigcup_{\underline{p}\in \varDelta_R(N_1,\ldots, N_d) \,\cap\, \Z^d} H(\underline{p})\bigg) \\
   &\geq \mathrm{vol}_{\R^d}\, \varDelta_R(N_1,\ldots, N_d).
  \end{align*}
  Moreover,
  \begin{align*}
   &\mathrm{vol}_{\R^{d}} \bigg(\bigcup_{\underline{p}\in \varDelta_R(N_1,\ldots, N_d) \,\cap\, \Z^d} H(\underline{p})\bigg) \\
   &\leq \mathrm{vol}_{\R^{d}} \{(x_1,\ldots,x_{d}) \in (\R_{\geq 0})^{d} ~|~ (x_1-1)N_1 + (x_2-1)N_2 + \cdots + (x_{d}-1)N_{d} \leq R \} \\
   &= \mathrm{vol}_{\R^{d}} \Bigg\{(x_1,\ldots,x_{d}) \in (\R_{\geq 0})^{d} ~\bigg|~ x_1 N_1 + x_2 N_2 + \cdots + x_{d} N_{d} \leq R + \sum_{i=1}^{d} N_i \Bigg\} \\
   &= \mathrm{vol}_{\R^{d}} \Delta_{R+\sum_{i=1}^{d} N_i}(N_1,\ldots,N_d),
  \end{align*}
  thus proving \eqref{vol2s}. For \eqref{Rd}, we have
  \begin{align*}
   \mathrm{vol}_{\R^d}\, \varDelta_R(N_1,\ldots, N_d) &= \int_{\substack{x_1N_1 + \ldots + x_{d}N_d \,\leq\, R \\ x_1,\,\ldots,\,x_{d} \,\geq\, 0}} \mathrm{d}x_1\cdots\mathrm{d}x_{d} \\
   &= \frac{R^{d}}{N_1\cdots N_d} \int_{\substack{x_1 + \ldots + x_{d} \,\leq\, 1 \\ x_1,\,\ldots,\,x_{d} \,\geq\, 0}} \mathrm{d}x_1\cdots\mathrm{d}x_{d} = \frac{1}{d!}\frac{R^{d}}{N_1\cdots N_d}. \qedhere
  \end{align*}
 \end{proof}
 
 \begin{lem}\label{Snn}
  Let $A = \{0 = a_0 < a_1 < \cdots < a_{\ell} < a_{\ell+1} =: m\}\subseteq \Z$ be a finite set of integers with $\gcd(A)=1$. For $n\geq 0$, let 
  \begin{equation*}
   S(n) := \bigg\{(\mu_2,\ldots, \mu_{\ell+1}) \in \{0,\ldots,a_1-1\}^{\ell} ~\bigg|~ \sum_{j=2}^{\ell+1} a_j \mu_j \equiv n \pmod{a_1}\bigg\}.
  \end{equation*}
  Then $|S(n)| = a_1^{\ell-1}$.
 \end{lem}
 \begin{proof}
  If $0\leq b<a_1$ is such that $n\equiv b \pmod{a_1}$ then $|S(n)|=|S(b)|$. Let $g:= \gcd(a_2,\ldots, a_{\ell},m)$ and $x_2,\ldots, x_{\ell+1}\in \Z$ be such that $\sum_{j=2}^{\ell+1} a_jx_j = g$. For each $(\mu_2,\ldots, \mu_{\ell+1}) \in S(b)$ we have $(\mu'_2,\ldots,\mu'_{\ell+1})\in S(b+g)$, where $0\leq \mu'_j < a_1$ is such that $\mu'_j \equiv \mu_j+x_j\pmod{a_1}$. Hence, $|S(b)| = |S(b+g)|$.
  
  By the definition of $g$, $\gcd(a_1,g) = 1$, so $\{b,b+g,\ldots,b+(a_1-1)g\}$ is a complete set of representatives of residues modulo $a_1$. Extending the argument above, we get that $|S(b)| = |S(b+g)| = \cdots = |S(b+(a_1-1)g)|$. Therefore, as $\sum_{b=0}^{a_1-1} |S(b)| = \#\{(\mu_2,\ldots, \mu_{\ell+1}) \in \{0,\ldots,a_1-1\}^{\ell}\} = a_1^{\ell}$, $|S(n)| = a_1^{\ell-1}$ for every $n\geq 0$.
 \end{proof}

 \begin{proof}[Proof of Proposition \ref{lbrho}]
  For each representation $k_1a_1 + \cdots +k_{\ell}a_{\ell} + k_{\ell+1}m = n$, we can write $k_j = \mu_j +a_1 q_j$ for some unique $\mu_{j}\in \{0,\ldots, a_1-1\}$ and $q_j\geq 0$, so that
  \[ k_1 + \sum_{j=2}^{\ell+1} a_j q_j = \frac{n - \sum_{j=2}^{\ell+1} a_j \mu_j}{a_1}. \]
  Thus,
  \begin{align*}
   r_A(n) &= \#\big\{(k_1,\ldots,k_{\ell+1}) \in (\Z_{\geq 0})^{\ell+1} ~\big|~ k_1a_1 + \cdots +k_{\ell}a_{\ell} + k_{\ell+1}m = n\big\} \\
   &= \sum_{(\mu_2,\ldots,\mu_{\ell+1})\in S(n)} r_{\{1,a_2,\ldots,a_{\ell}, m\}}\bigg( \frac{n - \sum_{j=2}^{\ell+1} a_j \mu_j}{a_1} \bigg),
  \end{align*}
  where $S(n)$ is as in Lemma \ref{Snn}. For each $N\geq 0$, in the notation of Lemma \ref{clp}, we have
  \begin{align*}
   r_{\{1,a_2,\ldots,a_{\ell},m\}}(N) &= \#\big\{(q_2,\ldots,q_{\ell+1}) \in (\Z_{\geq 0})^{\ell} ~\big|~ q_2a_2 +\cdots +q_{\ell}a_{\ell} + q_{\ell+1}m \leq N \big\} \\
   &= |\Delta_{N}(a_2,\ldots,a_{\ell+1}) \cap \Z^{\ell}|,
  \end{align*}
  so, by Lemma \ref{clp},
  \[ \frac{1}{\ell!} \frac{N^{\ell}}{a_2\cdots a_{\ell}m} \,\leq\, r_{\{1,a_2,\ldots,a_{\ell},m\}}(N) \,\leq\, \frac{1}{\ell!} \frac{(N+ \sum_{j=2}^{\ell+1} a_{j})^{\ell}}{a_2\cdots a_{\ell}m}. \]
  Therefore, since $0\leq n - (a_1-1)\sum_{j=2}^{\ell+1} a_j \leq n - \sum_{j=2}^{\ell+1} a_j \mu_j \leq n$, we have
  \begin{equation*}
   |S(n)|\cdot \frac{1}{\ell!} \frac{(n-(a_1-1)\sum_{j=2}^{\ell+1} a_j)^{\ell}}{a_1^{\ell}\, a_2\cdots a_{\ell}m} \,\leq\, r_A(n) \,\leq\, |S(n)|\cdot \frac{1}{\ell!} \frac{(n +  a_1\sum_{j=2}^{\ell+1} a_{j})^{\ell}}{a_1^{\ell}\, a_2\cdots a_{\ell}m}
  \end{equation*}
  so, by Lemma \ref{Snn}:
  \[ \frac{1}{\ell!} \frac{(n-(a_1-1)\sum_{j=2}^{\ell+1} a_j)^{\ell}}{a_1\cdots a_{\ell}m} \,\leq\, r_A(n) \,\leq\, \frac{1}{\ell!} \frac{(n + a_1\sum_{j=2}^{\ell+1} a_{j})^{\ell}}{a_1\cdots a_{\ell}m}. \qedhere \]
 \end{proof}
 
 \begin{rem}
  If $\Delta_A := (2a_1-1)\sum_{j=2}^{\ell+1}a_j$, then Proposition \ref{lbrho} shows, in particular, that $r_A(n + \Delta_A + k) \geq r_A(n)$ for every $k\geq 0$, $n \geq (a_1-1)\sum_{j=2}^{\ell+1}a_j$. In fact, if $N= n+\Delta_A + k$, then
  \begin{align*}
   r_A(N) &\geq \frac{1}{\ell!} \frac{(n + a_1\sum_{j=2}^{\ell+1} a_{j} + k)^{\ell}}{a_1 \cdots a_{\ell}m} \\
   &\geq \bigg(1+\frac{k\ell}{n + a_{1}\sum_{j=2}^{\ell+1} a_{j}}\bigg)\frac{1}{\ell!} \frac{(n + a_1\sum_{j=2}^{\ell+1} a_{j})^{\ell}}{a_1 \cdots a_{\ell}m} \\
   &\geq \bigg(1+\frac{k\ell}{n + \Delta_A}\bigg) r_A(n).
  \end{align*}
 \end{rem}
 
 \begin{cor}[Estimates for $\Fr_t(A)$]\label{lbrhoC}
  If 
  \[ n> (a_1\cdots a_{\ell}m)^{1/\ell}\, (\ell!)^{1/\ell} (t-1)^{1/\ell} + (a_1-1)\sum_{j=2}^{\ell+1} a_j, \]
  then $r_A(n) \geq t$, and if
  \[ n\leq (a_1\cdots a_{\ell}m)^{1/\ell}\, (\ell!)^{1/\ell} (t-1)^{1/\ell} - a_1\sum_{j=2}^{\ell+1} a_j - 1, \]
  then $r_A(n) < t$. In particular:
  \begin{itemize}
   \item $\Fr_t(A) \leq (a_1\cdots a_{\ell}m)^{1/\ell}\, (\ell!)^{1/\ell} (t-1)^{1/\ell} + (a_1-1)\sum_{j=2}^{\ell+1} a_j$;\smallskip
   
   \item $\Fr_t(A) > (a_1\cdots a_{\ell}m)^{1/\ell}\, (\ell!)^{1/\ell} (t-1)^{1/\ell} - a_1\sum_{j=2}^{\ell+1} a_j - 1$.
  \end{itemize}
 \end{cor}
 \begin{proof}
  Let $N\geq (a_1-1)\sum_{j=2}^{\ell+1} a_j$ be the largest integer such that
  \[ \frac{1}{\ell!} \frac{(N-(a_1-1)\sum_{j=2}^{\ell+1} a_j)^{\ell}}{a_1\cdots a_{\ell}m} \leq t-1. \] 
  This is equivalent to
  \begin{equation*}
   N \leq (a_1\cdots a_{\ell}m)^{1/\ell}\, (\ell!)^{1/\ell} (t-1)^{1/\ell} + (a_1-1)\sum_{j=2}^{\ell+1} a_j.
  \end{equation*}
  By Proposition \ref{lbrho}, $r_A(n) \geq t$ for all $n> N$, so $N \geq \Fr_t(A)$, proving the first part. Similarly, let $M\geq 0$ be the largest integer such that
  \[ \frac{1}{\ell!} \frac{(M + a_1\sum_{j=2}^{\ell+1} a_{j})^{\ell}}{a_1 \cdots a_{\ell}m} \leq t-1. \]
  That is given by
  \begin{align*}
   M &= \lfloor(a_1\cdots a_{\ell}m)^{1/\ell}\, (\ell!)^{1/\ell} (t-1)^{1/\ell}\rfloor -a_1\sum_{j=2}^{\ell+1} a_j \\
   &> (a_1\cdots a_{\ell}m)^{1/\ell}\, (\ell!)^{1/\ell} (t-1)^{1/\ell} -a_1\sum_{j=2}^{\ell+1} a_j -1.
  \end{align*}
  By Proposition \ref{lbrho}, $r_A(n) \leq t-1$ for every $n\leq M$, so $\Fr_t(A) \geq M$. 
 \end{proof}

\section{The structure theorem for \texorpdfstring{$(hA)^{(t)}$}{(hA)\^{}(t)}}\label{sec3}
 For this section, let $A=\{0 = a_0 < a_1 <\cdots < a_{\ell} < a_{\ell+1} = m\}\subseteq \Z$ be a finite set of integers with $\gcd(A) = 1$, and let $t\geq 1$ be a fixed integer.
 
 \begin{lem}\label{AheqA}
  $r_{A,h}(n) = r_{A}(n)$ for every $h\geq n/a_{1}$.
 \end{lem}
 \begin{proof}
  Since $0\in A$, we have $r_{A,1}(n) \leq r_{A,2}(n) \leq \cdots\leq r_{A}(n)$ for every fixed $n$. Let $c_1+\cdots+c_u = n$ (with $c_i\in A\setminus\{0\}$) be a generic representation of $n$. Since $c_i\geq a_1$, we have $n = c_1+\cdots+c_u \geq ua_1$, and so $u\leq n/a_1$. Thus, if $h\geq n/a_1$, then every representation is counted.
 \end{proof}
 
 \begin{lem}\label{Fpm}
  For every $k\geq 1$, we have 
  \[ \{\Fr_t(A)+1, \ldots,\Fr_t(A)+ km\} \subseteq ((H_{+}+k-1) A)^{(t)}, \]
  where $H_{+} := \lceil(\Fr_t(A)+m)/a_1\rceil$.
 \end{lem}
 \begin{proof}
  We have $\{\Fr_t(A)+1,\ldots, \Fr_t(A)+m\} \subseteq \P_t(A)$, so it follows by Lemma \ref{AheqA} that $\{\Fr_t(A)+1,\ldots, \Fr_t(A)+m\} \subseteq (H_{+} A)^{(t)}$. Then,
  \begin{align*}
   \{\Fr_t(A)+1,\ldots, \Fr_t(A)+km\} &= \{\Fr_t(A) + 1,\ldots,\Fr_t(A)+m\} + (k-1)\{0,m\} \\
   &\subseteq (H_{+} A)^{(t)} + (k-1)A \\
   &\subseteq ((H_{+}+k-1) A)^{(t)}.
  \end{align*}
  The last line follows from the fact that $A+(hA)^{(t)} \subseteq ((h+1)A)^{(t)}$ (since $r_{A,h}(n) \leq r_{A,h+1}(n+a)$ for any $a\in A\setminus\{0\}$).
 \end{proof}

%

 Throughout the rest of this section, let
 \[ H_{+} := \left\lceil\frac{\Fr_t(A)+m}{a_1}\right\rceil, \qquad H_{-} := \left\lceil\frac{\Fr_t(m-A)+m}{m-a_{\ell}}\right\rceil. \]

 \begin{lem}\label{Fpm2}
  For every $h\geq \max\{H_{+}, H_{-}\}$, we have 
  \begin{equation*}
   \begin{aligned}
    \Big(&\{0,1,\ldots, \Fr_t(A)+ (h-H_{+}+1)m \}\,\cup \\
    &\hspace{1em}\{ (H_{-}-1)m - \Fr_t(m-A), \ldots, hm \} \Big) \setminus \big(\mathcal{E}_t(A)\cup (hm -\mathcal{E}_t(m-A))\big) \subseteq (hA)^{(t)}.
   \end{aligned}
  \end{equation*}
 \end{lem}
 \begin{proof}
  Applying Lemma \ref{Fpm} with $k= h-H_{+}+1$, we have $\{\Fr_t(A)+1,\ldots,\Fr_t(A)+ (h-H_{+}+1)m\} \subseteq (h A)^{(t)}$. Since $h\geq \Fr_t(A)/a_1$, by Lemma \ref{AheqA} we also have $\{0,1,\ldots,\Fr_t(A)\}\setminus \mathcal{E}_t(A) = \{0,1,\ldots,\Fr_t(A)\}\cap \P_t(A) \subseteq (hA)^{(t)}$. Therefore:
  \begin{equation}
   \{0,1,\ldots, \Fr_t(A)+ (h-H_{+}+1)m \} \setminus \mathcal{E}_t(A) \subseteq (hA)^{(t)}. \label{side1}
  \end{equation}
  
  Applying the same argument above for $m-A$, with $k = h-H_{-}+1$, we get that $\{0,1,\ldots, \Fr_t(m-A)+ (h-H_{-}+1)m \} \setminus \mathcal{E}_t(m-A) \subseteq (h(m-A))^{(t)}$. Since $(hA)^{(t)} = hm - (h(m-A))^{(t)}$, this is equivalent to
  \begin{equation}
   \{(H_{-}-1)m - \Fr_t(m-A), \ldots, hm \} \setminus (hm -\mathcal{E}_t(m-A)) \subseteq (hA)^{(t)}. \label{side2}
  \end{equation}
  Putting \eqref{side1} and \eqref{side2} together yields the lemma.
 \end{proof}
 
 \begin{lem}\label{finlem}
  $H_{+} + H_{-} - 2 \geq \max\{H_{+},H_{-}\}$.
 \end{lem}
 \begin{proof}  
  Since $a_1 \leq m-1$, we have $2a_1 - m \leq a_1-1$. As $a_1 - 1\leq \Fr(A)$, it follows that $2 \leq (\Fr(A)+m)/a_1 \leq \lceil (\Fr_t(A)+m)/a_1\rceil = H_{+}$. Therefore $H_{+} + H_{-} - 2 \geq H_{-}$. Analogously, we deduce that $H_{-}\geq 2$, and so $H_{+} + H_{-} - 2 \geq H_{+}$.
 \end{proof}

 \subsection{Proof of Theorem \ref{MT1}}
  By Lemma \ref{Fpm2}, we have that if $h\geq \max\{H_{+}, H_{-}\}$ then $(hA)^{(t)}$ is structured, unless
  \begin{equation*}
   \Fr_t(A)+ (h-H_{+}+1)m < (H_{-}-1)m - \Fr_t(m-A),
  \end{equation*}
  which is equivalent to
  \begin{equation}
   h< (H_{+} + H_{-} - 2) - \frac{\Fr_t(A)}{m} - \frac{\Fr_t(m-A)}{m}. \label{shdnt}
  \end{equation}
  However, by hypothesis, using Lemma \ref{finlem} we have
  \begin{align*}
   h \geq \left\lfloor\frac{\Fr_t(A)+m}{a_1}\right\rfloor + \left\lfloor\frac{\Fr_t(m-A)+m}{m-a_{\ell}}\right\rfloor &\geq H_{+} + H_{-} - 2 \\
   &\geq \max\{H_{+},H_{-}\},
  \end{align*}
  so the condition of Lemma \ref{Fpm2} is satisfied and \eqref{shdnt} never occurs.\hfill$\square$
  
  \begin{rem}
   Our proof gives something more: For $h\geq \big\lfloor\frac{\Fr_t(A)+m}{a_1}\big\rfloor + \big\lfloor\frac{\Fr_t(m-A)+m}{m-a_{\ell}}\big\rfloor$, it shows that $(hA)^{(t)} = [hm]\setminus \big(\mathcal{E}_t(A) \cup (hm-\mathcal{E}_t(m-A))\big)$ contains an interval of length at least $2m$. Indeed, we have
   \[ \{\Fr_t(A)+1,\ldots, hm - \Fr_{t}(m-A)-1\} \subseteq (hA)^{(t)}, \]   
   so $(hA)^{(t)}$ contains an interval of length at least
   \begin{align*}
    (hm- \,&\Fr_t(m-A)) - \Fr_t(A) - 1 \\
    &\geq \left\lfloor\frac{\Fr_t(A)+m}{a_1}\right\rfloor m + \left\lfloor\frac{\Fr_t(m-A)+m}{m-a_{\ell}}\right\rfloor m - \Fr_t(m-A) - \Fr_t(A) - 1 \\
    &> \bigg(\frac{m}{a_{1}} - 1\bigg)(\Fr_t(A) + m) + \bigg(\frac{m}{m-a_{\ell}} - 1\bigg)(\Fr_t(m-A) + m) - 1 \\
    &\geq \bigg(\frac{m}{a_{1}} + \frac{m}{m-a_{\ell}} - 2\bigg)\, m - 1.
   \end{align*}
   Since $a_{\ell}\geq a_{1}$, we have $\frac{m}{a_1}+\frac{m}{m-a_{\ell}} \geq \frac{m}{a_1}+\frac{m}{m-a_{1}} = \frac{m^2}{a_1(m-a_1)} \geq 4$, the last inequality being minimized for $a_1 = m/2$. Thus,
   \[ (hm- \Fr_t(m-A)) - \Fr_t(A) - 1 > 2m - 1. \]
  \end{rem}

\subsection{Proof of Theorem \ref{MT2}}
 Plugging the estimate of Corollary \ref{lbrhoC} into Theorem \ref{MT1} yields that $(hA)^{(t)}$ is structured for every $h \geq h_t(A)$, where
 \begin{align}
  h_t(A) &\leq \frac{\Fr_t(A)+m}{a_1} + \frac{\Fr_t(m-A)+m}{m-a_{\ell}} \nonumber \\
  &\leq \Bigg(\frac{m}{a_1} + \frac{(a_2\cdots a_{\ell}m)^{1/\ell} (\ell!)^{1/\ell} (t-1)^{1/\ell}}{a_1^{1-1/\ell}} + \bigg(1-\frac{1}{a_1} \bigg)\sum_{j=2}^{\ell+1} a_j\Bigg)\ + \nonumber \\
  &\hspace{3em}+ \Bigg(\frac{m}{m-a_{\ell}} + \frac{(m(m-a_1)\cdots(m-a_{\ell-1}))^{1/\ell}}{(m-a_{\ell})^{1-1/\ell}} (\ell!)^{1/\ell} (t-1)^{1/\ell} \ + \nonumber \\
  &\hspace{17em}+\bigg(1-\frac{1}{m-a_{\ell}} \bigg)\sum_{j=0}^{\ell-1} (m-a_j)\Bigg). \nonumber
 \end{align}
 By the AM-GM inequality,
 \[ (a_2\cdots a_{\ell}m)^{1/\ell} \leq \frac{1}{\ell}\bigg(\sum_{j=2}^{\ell+1} a_j \bigg),\quad (m(m-a_1)\cdots(m-a_{\ell-1}))^{1/\ell} \leq \frac{1}{\ell}\bigg(\sum_{j=0}^{\ell-1} (m-a_{j}) \bigg), \]
 so for $t\geq 2$,
 \begin{align}
  h_t(A) &\leq \bigg(\frac{1}{a_1}+\frac{1}{m-a_{\ell}}\bigg) m + \bigg(\frac{(\ell!)^{1/\ell}}{\ell} \frac{(t-1)^{1/\ell}}{a_1^{1-1/\ell}} + 1-\frac{1}{a_1} \bigg)\sum_{j=2}^{\ell+1} a_j\ + \nonumber \\
  &\hspace{8em}+ \bigg(\frac{(\ell!)^{1/\ell}}{\ell} \frac{(t-1)^{1/\ell}}{(m-a_{\ell})^{1-1/\ell}} + 1-\frac{1}{m-a_{\ell}} \bigg)\sum_{j=0}^{\ell-1} (m-a_j) \nonumber \\
  &< 2m + \bigg(\frac{(\ell!)^{1/\ell}}{\ell} \frac{(t-1)^{1/\ell}}{\min\{a_1,m-a_{\ell}\}^{1-1/\ell}} + 1 \bigg)\Bigg(\sum_{j=2}^{\ell+1} a_j + \sum_{j=0}^{\ell-1} (m-a_j) \Bigg) \nonumber \\
  &< 2m + \bigg(\frac{(\ell!)^{1/\ell}}{\ell} \frac{(t-1)^{1/\ell}}{\min\{a_1,m-a_{\ell}\}^{1-1/\ell}} + 1 \bigg)(\ell+2)m \nonumber \\
  &= \underbrace{\Bigg(\frac{2e}{\ell t^{1/\ell}} + \Bigg(\frac{e(\ell!)^{1/\ell}/\ell}{\min\{a_1,m-a_{\ell}\}^{1-1/\ell}}\bigg(1-\frac{1}{t}\bigg)^{1/\ell} + \frac{e}{t^{1/\ell}}\Bigg)\bigg(1+\frac{2}{\ell}\bigg)\Bigg)}_{=:\, C_{A,t}}  \frac{1}{e} m\ell t^{1/\ell}.\nonumber
 \end{align}
 
 Using that $(\ell!)^{1/\ell} \leq \frac{1}{e}(\ell + \log 4 \ell)$ for $\ell\geq 2$, we have
 \begin{equation}
  C_{A,t} \leq \bigg(1+\frac{4}{\ell}\bigg)\frac{e}{t^{1/\ell}} + \bigg(1+\frac{2}{\ell}\bigg)\frac{1 + (\log 4\ell)/\ell}{\min\{a_1,m-a_{\ell}\}}, \label{dddd}
 \end{equation}
 completing the proof.\hfill$\square$
 
 \begin{rem}
  Note that if we take $\ell \to \infty$, $t^{1/\ell} \to \infty$ for sets $A\subseteq \Z_{\geq 0}$ with $\min\{a_1,m-a_{\ell}\} \geq k$ for some fixed $k\geq 1$, then $C_{A,t}\leq (1+o(1)) k^{-1}$.
 \end{rem}
 
 \begin{rem}[Comparison with Yang--Zhou]\label{compYZ}
  Suppose that $t \geq 8\ell \geq 32$. Yang--Zhou \cite{yanzho21} showed that $(hA)^{(t)}$ is structured as in \eqref{strt} if $ h \geq \sum_{i=2}^{\ell+1} (ta_i - 1) -1$. From the crude lower bound
  \begin{align*}
   \sum_{i=2}^{\ell+1} (ta_i - 1) -1 &\geq mt + \bigg(\frac{\ell(\ell+1)}{2} - 1 \bigg)t - \ell - 1 \\
   &=\bigg(m + \frac{\ell^2+\ell-2}{2} - \frac{\ell}{t} - \frac{1}{t} \bigg)t > \bigg(m+ \frac{\ell^2}{2} \bigg)t,
  \end{align*}
  we would like to show the inequality
  \begin{equation}
   C_{A,t}\cdot \frac{1}{e} m\ell t^{1/\ell} \leq \bigg(m+ \frac{\ell^2}{2}\bigg)t, \label{improv}
  \end{equation}
  or, equivalently,
  \begin{align*}
   C_{A,t}\cdot \frac{1}{e}\frac{t^{1/\ell}}{t} \leq \frac{1}{\ell} + \frac{\ell}{2m}.
  \end{align*}
  The inequality above is true if
  \[ t^{1-1/\ell} \geq \frac{C_{A,t}}{e} \ell, \]
  or $t\geq (C_{A,t}/e)^{1+\frac{1}{\ell-1}} \ell^{1+\frac{1}{\ell-1}}$. For $\ell \geq 4$, from the numerical estimates
  \[ \ell^{1+\frac{1}{\ell-1}} \leq (4^{\frac{1}{3}})\ell < 1.6\ell \quad\text{ and }\quad (C_{A,t}/e)^{1+\frac{1}{\ell-1}} \leq (3e/e)^{1+\frac{1}{\ell-1}} \leq 5, \]
  it follows that \eqref{improv} holds true for $t\geq 8\ell$.
 \end{rem}
 
\section{An extremal family of examples}\label{sec4}
 Let $m\geq 5$ be an integer. Take integers $2\leq \ell \leq m/2$ and $0\leq R \leq (m-\ell)/(\ell-1)$, and define
 \begin{equation*}
  A = A_{\ell,m} :=\{0,1,m-\ell+1,\ldots,m\},\qquad t = t_R := \binom{\ell + R}{R},
 \end{equation*}
 so that $|A|=\ell+2$. We are going to show that if $h_t = h_{t_R}(A_{\ell,m})$ is the smallest integer for which $(hA_{m,\ell})^{(t_R)}$ is structured for every $h\geq h_t$, then $h_t \geq (R+1)(m-\ell+1)-1$.
 
 \begin{lem}\label{excmR}
  $\mathcal{E}_t(A),\, \mathcal{E}_t(m-A) \subseteq [m R -1]$.
 \end{lem}
 \begin{proof}
  Since $1\in A \cap (m-A)$, it suffices to show that $r_A(mR)$, $r_{m-A}(mR) \geq t$. The representations of $mR$ are of the form
  \[ m R = k_0\cdot 1 + k_1 (m-\ell+1) + \cdots + k_{\ell-1}(m-1) + k_{\ell} m \]
  in $A$, and 
  \[ m R = k'_0\cdot 1 + \cdots + k'_{\ell-2}(\ell-1) + k'_{\ell-1}(m-1) + k'_{\ell} m \]
  in $m-A$, with $k_i$, $k'_i\geq 0$. For any choice of $k_i$'s ($i\geq 1$) with $\sum_{i=1}^{\ell} k_i \leq R$ (resp. $k'_i$), we can define $k_0\geq 0$ (resp. $k'_0$) by the difference. Thus,
  \[ r_{A}(mR) \geq \#\bigg\{(k_1,\ldots,k_\ell) \in \Z_{\geq 0}^\ell ~\bigg|~ \sum_{i=1}^{\ell} k_i \leq R \bigg\} = \binom{\ell+R}{R} = t, \]
  and similarly $r_{m-A}(mR) \geq t$.
 \end{proof}
 
 Throughout the rest of the argument, let
 \begin{equation}
  g := (R+1)(m-\ell+1)-1. \label{defg}
 \end{equation}
 Since $0< (R+1)(\ell-1) + 1\leq m$ for $R$ in our range, we have $\lfloor \frac{g}{m}\rfloor = \lfloor R + 1 - \frac{(R+1)(\ell-1) + 1}{m}\rfloor = R$. We can therefore write $g = x + mR$, for some $x\in[m-1]$.
 
 \begin{lem}\label{rhoAt}
  $r_A(g) = t$.
 \end{lem}
 \begin{proof}
  The representations of $g$ are of the form
  \begin{equation}
   g = k_0\cdot 1 + k_1 (m-\ell+1) + \cdots + k_{\ell-1}(m-1) + k_{\ell} m, \label{reppG}
  \end{equation}
  with $k_i\geq 0$. By the definition of $g$, there cannot be $i\geq 1$ for which $k_i > R$; in particular, we cannot have $\sum_{i=1}^{\ell} k_i > R$. At the same time, since $mR\leq g$, for all choices of $k_i$'s ($i\geq 1$) with $\sum_{i=1}^{\ell} k_i \leq R$, we can define $k_0\geq 0$ by the difference. Therefore:
  \[ r_{A}(g) = \#\bigg\{(k_1,\ldots,k_\ell) \in \Z_{\geq 0}^\ell ~\bigg|~ \sum_{i=1}^{\ell} k_i \leq R \bigg\} = \binom{\ell+R}{R} = t. \qedhere \]
 \end{proof}

 \begin{cor}
  $g-m \leq \Fr_t(A) \leq m\lfloor \frac{g}{m}\rfloor -1$.
 \end{cor}
 \begin{proof}
  Since $\lfloor \frac{g}{m}\rfloor = R$, Lemma \ref{excmR} implies the upper bound. For the lower bound, note that every representation of $g-m$ can be turned into a representation of $g$ by adding the element $m$; moreover, we have the representation $g= k_0\cdot 1$ in \eqref{reppG}, which is not obtained in this fashion. So $r_A(g) \geq r_A(g-m)+1$, which by Lemma \ref{rhoAt} implies that $r_{A}(g-m) \leq t-1$, and thus $g-m\in \mathcal{E}_t(A)$ (or $g-m\leq 0$, in which case the lower bound holds trivially).
 \end{proof}

 \begin{prop}\label{prop4}
  \phantom{.}
 
  \begin{enumerate}[label=\textnormal{(\roman*)}]
   \item $(HA)^{(t)} \neq [Hm]\setminus \big(\mathcal{E}_t(A)\cup (Hm-\mathcal{E}_t(m-A)) \big)$, for $H := 2R+1$.\smallskip
   
   \item If $h> 2R+1$ and $(hA)^{(t)} = [hm]\setminus \big(\mathcal{E}_t(A)\cup (hm-\mathcal{E}_t(m-A)) \big)$, then we must have $h\geq g = (R+1)(m-\ell+1)-1$.
  \end{enumerate}
 \end{prop}
 \begin{proof}
  The representation $g= k_0\cdot 1$ in \eqref{reppG} is the one with the most summands, so $r_{A,h}(g) = t$ if and only if $h\geq g = (R+1)(m-\ell+1)-1$. However, by Lemma \ref{excmR}, we know that
  \[ g \in x + m\cdot \{R, R+1,\ldots, H - 1 - R\} \subseteq [Hm] \setminus \big(\mathcal{E}_t(A)\cup (Hm-\mathcal{E}_t(m-A)) \big) \]
  (where $g = x+mR$, $x\in[m-1]$) for $H = 2R+1$; so since $2R+1 < (R+1)(m-\ell+1) - 1$, it follows that $g\notin (HA)^{(t)}$. We conclude that $(HA)^{(t)}$ cannot be structured, proving part (i). Part (ii) likewise follows from these arguments.
 \end{proof}
 
 \begin{rem}\label{explca}
  In Proposition \ref{prop4} we showed that $h_t(A) \geq g$. Since $t = \binom{\ell+R}{R}$, we have $\frac{(\ell+R)^{\ell}}{\ell^{\ell}} \leq t \leq \frac{(\ell+R)^{\ell}}{\ell!}$, so $(\ell!)^{1/\ell} t^{1/\ell}-\ell \leq R \leq \ell t^{1/\ell}-\ell$ and $t^{1/\ell} \geq 1+\frac{R}{\ell}$. Letting $m\to\infty$, with $\ell := \lfloor m^{1/2.01}\rfloor$ and $R=\lfloor m^{1/2}\rfloor$, we get $t^{1/\ell} \to \infty$. Thus (using that $\ell! \geq \ell^\ell/e^{\ell}$):
  \begin{align*}
   g = (R+1)(m-\ell +1) -1 &\geq ((\ell!)^{1/\ell} t^{1/\ell}-\ell)(m-\ell) \\
   &\geq (1-o_{m\to\infty}(1))\frac{1}{e} \ell t^{1/\ell}\cdot (1-o_{m\to\infty}(1))\,m \\
   &= (1 - o_{m\to\infty}(1))\, \frac{1}{e}m\ell t^{1/\ell}.
  \end{align*}
  Applying Theorem \ref{MT2}, we get $h_t(A) \sim_{m\to\infty} \dfrac{1}{e}m\ell t^{1/\ell}$.
 \end{rem}


 \section{\texorpdfstring{$t$}{t}-representables in \texorpdfstring{$\Z^d$}{Z\^{}d}}\label{defdelta}
 Let $\ul{0}\in A\subseteq \Z^d$ be as in Subsection \ref{ssec13}. Write $r_A(\ul{p}) := \lim_{h\to\infty} r_{A,h}(\ul{p})$ for the \emph{total representation function} of $A$, so that $\P_t(A) = \{\ul{p}\in \mathcal{C}_{A}\cap \Lambda_{A} ~|~ r_A(\ul{p}) \geq t\}$.

\subsection{Deduction of Theorem \ref{ZdFrobt} from Lemma \ref{mainlm}}
 Let $\ul{\mathfrak{n}}$, $\delta_A$, and $\Delta_A$ be as in \eqref{defd}. We start with a $\Z^d$ version of Lemma \ref{AheqA}.
 
 \begin{lem}\label{rhostab}
  If $\ul{p}\in \mathcal{C}_A\cap\Lambda_A$, then $r_{A,h}(\ul{p}) = r_{A}(\ul{p})$ for every $h\geq  \langle\ul{p},\ul{\mathfrak{n}} \rangle/\delta_{A}$.
 \end{lem}
 \begin{proof}
  Since $\ul{0}\in A$, we have $r_{A,1}(\ul{p}) \leq r_{A,2}(\ul{p}) \leq \cdots \leq r_A(\ul{p})$ for every fixed $\ul{p}\in \mathcal{C}_A\cap\Lambda_A$. Let $\ul{p} = \sum_{\ul{a}\in A\setminus\{0\}} k_{\ul{a}} \ul{a}$ (with $k_{\ul{a}}\in\Z_{\geq 0}$) be an arbitrary representation of $\ul{p}$ by $A$. Since $\langle \,\cdot\,,\ul{\mathfrak{n}} \rangle$ is linear, we have 
  \[ \langle \ul{p},\ul{\mathfrak{n}} \rangle = \sum_{\ul{a}\in A\setminus\{0\}} k_{\ul{a}} \langle \ul{a},\ul{\mathfrak{n}} \rangle \geq \delta_{A} \sum_{\ul{a}\in A\setminus\{0\}} k_{\ul{a}}, \]
  and thus $\sum_{\ul{a}\in A\setminus\{0\}} k_{\ul{a}} \leq \langle \ul{p},\ul{\mathfrak{n}} \rangle/\delta_A$. So every possible representation of $\ul{p}$ by $A$ is taken into account once $h\geq \langle \ul{p},\ul{\mathfrak{n}} \rangle/\delta_A$.
 \end{proof}

 Fix $t\geq 1$, and define
 \[ H_A = H_{A,t} := \left\lceil\frac{\Delta_{A}}{\delta_{A}} \varphi_{A,t}\right\rceil\quad \bigg(\text{resp. } H_{\ul{v}-A} := \left\lceil\frac{\Delta_{\ul{v}-A}}{\delta_{\ul{v}-A}} \varphi_{\ul{v}-A,t}\right\rceil, \text{ for }\ul{v}\in \ex(\cH(A))\bigg), \]
 where $\varphi_{A,t}$ (resp. $\varphi_{\ul{v}-A}$) is the constant from Lemma \ref{mainlm}. Thus, by Lemma \ref{rhostab},
 \begin{equation}
  (\varphi_{A,t} \cH(A) \cap \Lambda_A)\setminus \mathcal{E}_t(A) \subseteq (H_A A)^{(t)} \label{jej}
 \end{equation}
 (resp. $(\varphi_{\ul{v}-A,t} \cH(\ul{v}-A) \cap \Lambda_A)\setminus \mathcal{E}_t(\ul{v}-A) \subseteq (H_{\ul{v}-A} (\ul{v}-A))^{(t)}$). More generally: 

 \begin{lem}\label{sttpnt}
  For each $\ul{v}\in\mathrm{ex}(\cH(A))$ and $h\geq H_{\ul{v}-A}$, we have
  \[ \big(G_{\ul{v}}\,\ul{v} + (h -G_{\ul{v}})\cH(A)\big) \cap \Lambda_A\setminus \big(h\ul{v} - \mathcal{E}_t(\ul{v}-A)\big) \subseteq (hA)^{(t)}, \] 
  where $G_{\ul{v}} = G_{\ul{v},t} := H_{\ul{v}-A} - \varphi_{\ul{v}-A,t}$.
 \end{lem}
 \begin{proof}
  Applying Lemma \ref{mainlm} to $\lambda = \varphi_{\ul{v}-A,t} + k$ together with \eqref{jej}, we get
  \begin{align*}
   \big((\varphi_{\ul{v}-A,t}+k)\cH(\ul{v}-A)\cap \Lambda_A\big) &\setminus \mathcal{E}_t(\ul{v}-A) \\
   &= \big(\varphi_{\ul{v}-A,t}\cH(\ul{v}-A)\cap \Lambda_A\big) \setminus \mathcal{E}_t(\ul{v}-A) + k(\ul{v}-A) \\
   &\subseteq (H_{\ul{v}-A}(\ul{v}-A))^{(t)} + k(\ul{v}-A) \\
   &\subseteq ((H_{\ul{v}-A}+k) (\ul{v}-A))^{(t)}
  \end{align*}  
  for every integer $k\geq 0$. Taking $k := h-H_{\ul{v}-A}$, it follows that
  \begin{align*}
   \big((h -(H_{\ul{v}-A} - \varphi_{\ul{v}-A,t}))\cH(\ul{v}-A)\big)\cap \Lambda_A \setminus \mathcal{E}_t(\ul{v}-A) &\subseteq (h(\ul{v}-A))^{(t)} \\
   &= h\ul{v}- (hA)^{(t)}.
  \end{align*}
  Using that $\cH(\ul{v}-A) = \ul{v} - \cH(A)$, we have
  \[ \big(h\ul{v} - G_{\ul{v}}\ul{v} - (h - G_{\ul{v}})\cH(A)\big)\cap \Lambda_A \setminus \mathcal{E}_t(\ul{v}-A) \subseteq h\ul{v}- (hA)^{(t)}. \]
  Subtracting both sides from $h\ul{v}$ proves the lemma.
 \end{proof}
 
 Using this, we can show that certain parts of the structured set in \eqref{strzd} do belong to $(hA)^{(t)}$. 
 
 \begin{lem}\label{eachB}
  Let $\ul{0}\in B\subseteq \ex(\cH(A))$ be a subset such that $B\setminus \{\ul{0}\}$ is linearly independent and $|B| = d+1$. Then,
  \[ (h \cH(B) \cap \Lambda_{A}) \setminus \bigg(\mathcal{E}_t(A)\cup \bigcup_{\ul{b}\in B\setminus\{\ul{0}\}}\, (h\ul{b} - \mathcal{E}_t(\ul{b}-A)) \bigg) \subseteq (hA)^{(t)} \]
  for every $h \geq H_{A} + \sum_{\ul{b}\in B\setminus\{\ul{0}\}} H_{\ul{b} - A}$.
 \end{lem} 
 \begin{proof}
  Write $B = \{\ul{0},\ul{b}_1,\ldots, \ul{b}_d\}$, and $(x_1,\ldots,x_{d})_B := \sum_{i=1}^{d} x_i\ul{b}_i$. Take
  \begin{equation}
   \ul{p}\in (h \cH(B) \cap \Lambda_{A}) \setminus \bigg(\mathcal{E}_t(A)\cup \bigcup_{i=1}^{d}\, (h\ul{b}_i - \mathcal{E}_t(\ul{b}_i-A)) \bigg) \label{whichp}
  \end{equation}
  so that $\ul{p} = (p_1,\ldots,p_d)_B$ for some $\ul{p}_i\in \R_{\geq 0}$.
  
  Let $h\geq \max_{\ul{b}\in B\setminus\{\ul{0}\}} H_{\ul{b}-A}$. If there is $1\leq i\leq d$ such that $p_i \geq H_{\ul{b}_i - A} - \varphi_{\ul{b}_i - A,t}$, then
  \begin{equation*}
   \ul{p} \in (H_{\ul{b}_i-A} - \varphi_{\ul{b}_i-A,t})\,\ul{b}_i + (h -(H_{\ul{b}_i-A} - \varphi_{\ul{b}_i-A,t}))\cH(B),
  \end{equation*}
  which by Lemma \ref{sttpnt} implies that $\ul{p}\in (hA)^{(t)}$.
  
  Otherwise, $\ul{p}$ is such that $0\leq p_i \leq H_{\ul{b}_i - A} - \varphi_{\ul{b}_i - A,t}$ for every $1\leq i\leq d$. For each $i$, if $h\geq  H_{\ul{b}_i - A} - \varphi_{\ul{b}_i - A,t}$ then $ j\ul{b}_i \in h\cH(B)$ for every $0\leq j\leq H_{\ul{b}_i - A} - \varphi_{\ul{b}_i - A,t}$. Therefore, if $h\geq \sum_{i=1}^{d} (H_{\ul{b}_i - A} - \varphi_{\ul{b}_i - A,t})$, then
  \[ \ul{p} \in \{(x_1,\ldots, x_d)_B ~|~ 0\leq x_i \leq H_{\ul{b}_i - A} - \varphi_{\ul{b}_i - A,t},\, x_i\in \R  \}\cap \Lambda_A \subseteq h\cH(B) \cap \Lambda_A. \]
  Thus, for $h\geq H_A + \sum_{i=1}^{d} H_{\ul{b}_i - A}$, Lemma \ref{sttpnt} (applied with $\ul{v}=0$) guarantees that every $\ul{p}$ in \eqref{whichp} is in $(hA)^{(t)}$.
 \end{proof}
 
%

 Finally, to glue together all the parts of the structured set as they appear in Lemma \ref{eachB}, we use the following version of a classical result in convex geometry:
 
 \begin{lem}[Carath\'eodory's theorem]\label{caratlm}
  We have
  \[ \cH(A) = \bigcup_{\substack{\ul{0}\,\in\, B\, \subseteq\, \mathrm{ex}(\cH(A)) \\ |B|\,=\,d+1 \\ \mathrm{span}_{\R}(B) \,=\, \R^d}}\ \cH(B). \]
 \end{lem}
  
 
 \begin{proof}
  We adapt the proof of Granville--Shakan \cite[Lemma 4]{grasha20}. Since $\cH(\ex(\cH(A))) = \cH(A)$,\footnote{cf. Br{\o}ndsted \cite[Theorem 7.2]{brond83}.} we can suppose without loss of generality that $A = \ex(\cH(A))$. For any $\ul{v}\in \cH(A)$, we can represent it as $\ul{v} = \sum_{\ul{a}\in A} c_{\ul{a}} \ul{a}$ for some real $c_{\ul{a}} \geq 0$ such that
  \[ 0\leq \sum_{\ul{a}\in A} c_{\ul{a}} \leq 1. \]
  Define $B := \{\ul{a}\in A ~|~ c_{\ul{a}} > 0\}$, and select a representation $(c_{\ul{a}})_{\ul{a}\in A}$ that minimizes $|B|$. We claim that $B$ is a set of linearly independent elements.
  
  Indeed, suppose the contrary; i.e., that $\sum_{\ul{b}\in B} e_{\ul{b}} \ul{b} = 0$ for some set of $e_{\ul{b}} \in \R$ with not all $e_{\ul{b}} =0$. Without loss of generality, suppose that $\sum_{\ul{b}\in B} e_{\ul{b}} \geq 0$ (otherwise, change the sign of each $e_{\ul{b}}$). Hence, there exists at least one element of $\ul{b} \in B$ such that $e_{\ul{b}} > 0$. Define
  \[ m := \min_{\substack{\ul{b}\in B \\ e_{\ul{b}} >0}}\ c_{\ul{b}}/e_{\ul{b}}, \]
  so that $c_{\ul{b}^{*}} = m e_{\ul{b}^{*}}$ for some $\ul{b}^{*} \in B$, and $0 \leq \sum_{\ul{b}\in B} me_{\ul{b}} \leq \sum_{\ul{b}\in B} c_{\ul{b}}$. But then $\ul{v} = \sum_{\ul{b}\in B} (c_{\ul{b}} - me_{\ul{b}}) \ul{b}$, where $c_{\ul{b}} - me_{\ul{b}} \geq 0$ for every $\ul{b}\in B$, $0 \leq \sum_{\ul{b}\in B} (c_{\ul{b}} - me_{\ul{b}}) \leq 1$, and $c_{\ul{b}^{*}} - me_{\ul{b}^{*}} = 0$. This contradicts the minimality of $|B|$.
  
  Since we can then adjoin elements from $A$ to $B$ until we have $d$ independent non-zero elements, this concludes the proof.
 \end{proof}

 \begin{proof}[Proof of Theorem \ref{ZdFrobt}]
  By Lemma \ref{caratlm}, we have $h\cH(A) = \bigcup_{B} h\cH(B)$. Thus, taking $h \geq H_A + \max_{B} \sum_{\ul{b}\in B\setminus\{\ul{0}\}} H_{\ul{b} - A}$ we may apply Lemma \ref{eachB} to each $B$, obtaining the inclusion ``$\supseteq$'' in \eqref{strzd}. The inclusion ``$\subseteq$'' follows by definition.
 \end{proof}

 \begin{rem}[Analogy to the case $d=1$]\label{remana}
  Let
  \[ A=\{0 = a_0 < a_1 <\cdots < a_{\ell} < a_{\ell+1} = m\} \subseteq \Z \]
  be a set with $\gcd(A) = 1$. In this case, $\Lambda_A=\Z$, $\delta_A = a_1$, $\Delta_A= m$, $\cH(A) = [0,m]$, and $\mathrm{ex}(\cH(A)) = \{0,m\}$. By the definition of $\Fr_t(A)$, writing $F := (\Fr_t(A)+m)/m$, we have $\mathcal{E}_t(A) \subseteq \{0,\ldots,(F-1)m\}$, and so
  \[ \{\Fr_t(A)+1,\ldots,\Fr_t(A)+m\} \subseteq (F\cH(A)\cap \Lambda_A )\setminus \mathcal{E}_t(A). \]
  Thus, for every real $\lambda \geq F$, we have
  \begin{align*}
   (\lambda \cH(A)\cap \Lambda_A )\setminus \mathcal{E}_t(A) + A &= \{0,\ldots,\lfloor\lambda m\rfloor\} \cap \mathcal{P}_t(A) + A \\
   &= \{0,\ldots, \lfloor\lambda m\rfloor + m\} \cap \mathcal{P}_t(A) \\
   &= ((\lambda+1)\cH(A)\cap \Lambda_A )\setminus \mathcal{E}_t(A),
  \end{align*}
  which implies $\varphi_{A,t} \leq F$ as in Lemma \ref{mainlm}. Applying Theorem \ref{ZdFrobt}, we obtain that $(hA)^{(t)}$ is structured for every $h\geq h_t(A)$, with
  \begin{align*}
   h_t(A) &\leq \left\lceil\frac{\Delta_{A}}{\delta_{A}}\, \varphi_{A,t}\right\rceil + \left\lceil\frac{\Delta_{m-A}}{\delta_{m-A}}\, \varphi_{m-A,t}\right\rceil = \left\lceil\frac{\Fr_t(A)+m}{a_1}\right\rceil + \left\lceil\frac{\Fr_t(m-A) + m}{m-a_{\ell}}\right\rceil.
  \end{align*}
 \end{rem}

\subsection{Proof of Lemma \ref{mainlm}}
 Given $\ub{u} = (u_1,\ldots,u_d)$, $\ub{v} = (v_1,\ldots,v_d) \in \Z_{\geq 0}^{d}$, write $\ub{u}\leq \ub{v}$ is $u_i\leq v_i$ for every $1\leq i\leq d$. We will use the following classical lemma.
 
 \begin{lem}[Dickson's lemma]\label{mannlm}
  Let $S\subseteq \Z_{\geq 0}^{d}$. There is a finite subset $T\subseteq S$ such that for all $\ub{s}\in S$, there exists $\ub{t}\in T$ such that $\ub{t}\leq \ub{s}$.
 \end{lem}
 \begin{proof}
  See Lemma 5 in Granville--Shakan \cite{grasha20}.
 \end{proof}
 
 The next lemma is based on Proposition 4 of Granville--Shakan \cite{grasha20}.
 
 \begin{lem}\label{granlem}
  Let $\ul{0}\in B\subseteq \ex(\cH(A))$ be a subset such that $B\setminus \{\ul{0}\}$ is linearly independent and $|B| = d+1$. Then, there exists a finite set $F_B = F_{B,t}(A)\subseteq \mathcal{P}_t(A)$ such that
  \[ \mathcal{P}_t(A) \cap \mathcal{C}_B = F_B + \mathcal{P}(B) \]
 \end{lem}
 \begin{proof}
  The fundamental domain for $\Lambda_B$ is
  \[ \R^d/\Lambda_B \simeq \mathcal{F} := \bigg\{\sum_{\ul{b}\in B} s_{\ul{b}} \ul{b} \ \,\bigg|\ s_{\ul{b}} \in [0,1),\ \forall \ul{b}\in B \bigg\}. \]
  Since $\mathcal{F}$ is bounded, the set $\mathcal{F} \cap \Lambda_A$ is finite. Thus, we can partition $\mathcal{P}_t(A) \cap \mathcal{C}_{B}$ as
  \[ \mathcal{P}_t(A) \cap \mathcal{C}_B = \bigsqcup_{\ul{v}\in \mathcal{F} \cap \Lambda_A} (\ul{v}+\Lambda_B) \cap \mathcal{P}_t(A). \]
  
  For each $\ul{v}\in \mathcal{F} \cap \Lambda_A$, define
  \[ S_{\ul{v}} := \bigg\{ \ub{s} = (s_{\ul{b}})_{\ul{b}\in B\setminus\{\ul{0}\}} \in \Z_{\geq 0}^{d} \ \,\bigg|\ \ul{v} + \sum_{\ul{b}\in B\setminus\{\ul{0}\}} s_{\ul{b}} \ul{b} \in (\ul{v}+\Lambda_B) \cap \mathcal{P}_t(A)  \bigg\}. \]
  By Dickson's Lemma \ref{mannlm}, there is a finite subset $T_{\ul{v}}\subseteq S_{\ul{v}}$ such that for every $\ub{s}\in S_{\ul{v}}$, there is $\ub{t}\in T_{\ul{v}}$ such that $\ub{t}\leq \ub{s}$. Define the finite set
  \[ F_{\ul{v}} := \bigg\{ \ul{v} + \sum_{\ul{b}\in B\setminus\{\ul{0}\}} t_{\ul{b}} \ul{b} \ \,\bigg|\ \ub{t} = (t_{\ul{b}})_{\ul{b}\in B\setminus\{\ul{0}\}} \in T_{\ul{v}} \bigg\}. \]
  By definition, for every $\ul{p} \in (\ul{v}+\Lambda_B) \cap \mathcal{P}_t(A)$ there is $\ul{q}\in F_{\ul{v}}$ such that $\ul{p} - \ul{q} \in \mathcal{P}(B)$. Hence, $(\ul{v}+\Lambda_B) \cap \mathcal{P}_t(A) = F_{\ul{v}} + \mathcal{P}(B)$. Taking the union over $\ul{v}$'s, we conclude that
  \begin{align*}
   \mathcal{P}_t(A) \cap \mathcal{C}_B &= \bigsqcup_{\ul{v}\in \mathcal{F} \cap \Lambda_A} \big(F_{\ul{v}} + \mathcal{P}(B) \big) = F_{B} + \mathcal{P}(B),
  \end{align*}
  where $F_{B} := \bigsqcup_{\ul{v}\in \mathcal{F} \cap \Lambda_A} F_{\ul{v}}$.
 \end{proof}
 
 \begin{cor}\label{grancor}
  Let $\ul{0}\in B\subseteq \ex(\cH(A))$ be a subset such that $B\setminus \{\ul{0}\}$ is linearly independent and $|B| = d+1$, and let $\mu_{B} := \min\{\mu\in\R_{\geq 1} ~|~ F_B \subseteq \mu \cH(B)\}$, where $F_B$ is as in Lemma \ref{granlem}. Then, for real $\lambda \geq \mu_{B}$, we have
  \[ \mathcal{P}_t(A) \cap (\lambda +1)\mathcal{H}(B) = (\mathcal{P}_t(A) \cap \lambda\cH(B)) + B. \]
 \end{cor}
 \begin{proof}
  The inclusion ``$\supseteq$'' is clear. Since $\ul{0}\in B$, it suffices to show that for
  \[ \ul{p} \in \mathcal{P}_t(A) \cap (\lambda+1)\cH(B) \setminus \lambda\cH(B) \]
  there exists $\ul{b} \in B$ such that $\ul{p}-\ul{b}\in \mathcal{P}_t(A) \cap \lambda\cH(B)$.
 
  By Lemma \ref{granlem} we have $\mathcal{P}_t(A) \cap \lambda \cH(B) = (\mathcal{P}_t(A) \cap \mathcal{C}_B) \cap \lambda \cH(B) = (F_B + \mathcal{P}(B)) \cap \lambda \cH(B)$, so we can write $\ul{p} = \ul{f} + \sum_{\ul{b}\in B\setminus\{\ul{0}\}} c_{\ul{b}}\ul{b}$ for some $\ul{f}\in F_B$ and some integers $c_{\ul{b}} \in\Z_{\geq 0}$, with at least one (say, $\ul{b}^{*} \in B$) $c_{\ul{b}^{*}} \geq 1$, because $\lambda \geq \mu_B$. Thus
  \[ \ul{p}-\ul{b}^{*} = \ul{f} + \sum_{\ul{b}\in B\setminus\{\ul{0}\}} c_{\ul{b}}\ul{b} - \ul{b}^{*} \in F_{B}+\mathcal{P}(B) = \mathcal{P}_t(A)\cap \mathcal{C}_{B}, \]
  so it suffices to show that $\ul{p}-\ul{b}^{*}\in \lambda\cH(B)$. 

  Since $B\setminus\{\ul{0}\}$ is a basis for $\R^{d}$ and $\ul{p} \in (\lambda+1) \cH(B)\setminus\lambda\cH(B) \subseteq \mathcal{C}_B$, there exist unique $x_{\ul{b}}\in\R$, necessarily non-negative, such that $\ul{p} = \sum_{\ul{b}\in B\setminus\{0\}} x_{\ul{b}} \ul{b}$ and $\lambda < \sum_{\ul{b}\in B\setminus\{0\}} x_{\ul{b}} \leq \lambda+1$. Since $\ul{p}-\ul{b}^{*}\in \mathcal{C}_B$, there exist unique $y_{\ul{b}}\in \R$, necessarily non-negative, such that $\sum_{\ul{b}\in B\setminus\{0\}} y_{\ul{b}} \ul{b} = \ul{p}-\ul{b}^{*}$. By uniqueness, $y_{\ul{b}} = x_{\ul{b}}$ for $\ul{b}\in B\setminus\{\ul{b}^{*}\}$, and $y_{\ul{b}^{*}} = x_{\ul{b}^{*}} - 1$. Therefore, $\sum_{\ul{b}\in B\setminus\{0\}} y_{\ul{b}} \leq \lambda$, implying that $\ul{p}-\ul{b}^{*}\in \lambda\cH(B)$.
 \end{proof}
 
 \begin{proof}[Proof of Lemma \ref{mainlm}]
  Applying Lemma \ref{caratlm} together with Corollary \ref{grancor}, we obtain, for $\lambda \geq \max_{B} \mu_{B}$,
  \begin{align}
   \mathcal{P}_t(A) \cap (\lambda+1)\cH(A) &= \bigcup_{B}\ \mathcal{P}_t(A) \cap (\lambda+1)\cH(B) \nonumber \\
   &= \bigcup_{B}\ B+ (\mathcal{P}_t(A) \cap \lambda\cH(B)) \nonumber \\
   &\subseteq A+ \bigcup_{B}\ \mathcal{P}_t(A) \cap \lambda\cH(B) = A+(\mathcal{P}_t(A) \cap \lambda\cH(A)), \nonumber
  \end{align}
  where $\ul{0}\in B\subseteq \ex(\cH(A))$ runs through subsets such that $B\setminus \{\ul{0}\}$ is linearly independent and $|B| = d+1$. Since the reverse inclusion is trivial, writing $\mathcal{P}_t(A)\cap \lambda \cH(A) = (\lambda\cH(A)\cap\Lambda_A)\setminus\mathcal{E}_t(A)$ finishes the proof.
 \end{proof}
 
 \subsection{Proof of Theorem \ref{tKhov}}
 We generalize the method of Nathanson--Rusza \cite{natruz02}. Let $A = \{\ul{a}_{1},\ldots,\ul{a}_{\ell}\}\subseteq \Z^{d}$ be a finite subset of $\Z^d$. Define the set ${\widetilde{A}} = \{\ul{\alpha}_1,\ldots,\ul{\alpha}_{\ell}\} \subseteq \Z^{d+1}$ where each $\ul{\alpha}_j := (\ul{a}_j, 1)$. Define the projection
 \begin{align*}
  \pi = \pi_{\widetilde{A}}: \Z^{\ell} &\longrightarrow \Z^{d+1} \\
  \ub{u} &\longmapsto \sum_{i=1}^{\ell} {u}_i \ul{\alpha}_i,
 \end{align*}
 so that
 \begin{equation*}
  (hA)^{(t)} = \big\{ \ul{p}\in \Z^{d} ~\big|~ |\pi^{-1}(\ul{p},h)| \geq t \big\}. 
 \end{equation*}
 We define two orders in $\Z^{\ell}$:\smallskip
 \begin{itemize}
  \item $\ub{u} \leq \ub{v}$ if $\ub{v}-\ub{u} \in \Z_{\geq 0}^{\ell}$.\smallskip
  
  \item $\ub{u} \leq_{\mathrm{lex}} \ub{v}$ if there is $1\leq j\leq d$ such that $u_1=v_1$, $\ldots$, $u_{j-1} = v_{j-1}$, $u_{j}\leq v_{j}$. Note that this is a total ordering, and $\ub{u}\leq \ub{v}$ implies $\ub{u} \leq_{\mathrm{lex}} \ub{v}$ but not vice versa.\smallskip
 \end{itemize}
 
 For each $t\geq 1$, write
 \begin{equation}
  \mathcal{U}^{(t)} = \mathcal{U}_{\widetilde{A}}^{(t)} := \{ \ub{z}\in \Z^{\ell}_{\geq 0} ~|~ \exists \ub{w}_1,\ldots,\ub{w}_t\neq \ub{z} \text{ s.t. } \ub{w}_i<_{\mathrm{lex}} \ub{z},\ \ub{z}-\ub{w}_i \in \mathrm{ker}\,\pi\}, \label{Uideal}
 \end{equation}
 for the set of \emph{$t$-useless elements}, so that
 \begin{equation} 
 \begin{aligned}
  |(hA)^{(t)}| &= \#\{ \ub{u}\in \Z_{\geq 0}^{\ell} ~|~ \ub{u} \notin \mathcal{U}^{(t)} \text{ and } (\pi(\ub{u}))_{d+1} = h\} \ -\\
  &\hspace{5em} \#\{ \ub{u}\in \Z_{\geq 0}^{\ell} ~|~ \ub{u} \notin \mathcal{U}^{(t-1)} \text{ and } (\pi(\ub{u}))_{d+1} = h\}.
 \end{aligned}\label{hAt}
 \end{equation}
 Indeed, this counts those $\ub{u} \in \Z^{\ell}_{\geq 0}$ with $(\pi(\ub{u}))_{d+1} = h$ that have exactly $t-1$ elements $\ub{0}\leq \ub{v} <_{\mathrm{lex}} \ub{u}$ with $\pi(\ub{v}) = \pi(\ub{u})$. The set of $\leq$-minimal elements of $\mathcal{U}^{(t)}$
 \[ \mathcal{M}^{(t)} = \mathcal{M}_{\widetilde{A}}^{(t)} := \{ \ub{z}\in \mathcal{U}^{(t)} ~| \not\hspace{-.2em}\exists \ub{w}\in\mathcal{U}^{(t)} \text{ s.t. } \ub{w} < \ub{z}\}, \]
 is, by Dickson's Lemma \ref{mannlm}, a finite set. Note that $\ub{u}\in\mathcal{U}^{(t)}$ if and only if there exists $\ub{t}\in \mathcal{M}^{(t)}$ such that $\ub{u} \geq \ub{t}$.
 
 We claim that
 \begin{align}
  &\#\{ \ub{u}\in \Z_{\geq 0}^{\ell} ~|~ \ub{u} \notin \mathcal{U}^{(t)} \text{ and } (\pi(\ub{u}))_{d+1} = h\} \nonumber \\
  &\hspace{5em}= \#\{ \ub{u}\in \Z_{\geq 0}^{\ell} ~|~ (\pi(\ub{u}))_{d+1} = h\} - \#\{ \ub{u}\in \mathcal{U}^{(t)} ~|~ (\pi(\ub{u}))_{d+1} = h\} \nonumber \\
  &\hspace{5em}= \sum_{T\subseteq \mathcal{M}^{(t)}} (-1)^{|T|}\, \#\{\ub{u}\in \Z^{\ell}_{\geq 0} ~|~ (\pi(\ub{u}))_{d+1} = h \text{ and } \ub{t}_{T} \leq \ub{u}\}, \label{incexc}
 \end{align}
 where the sum runs over all subsets $T$ of $\mathcal{M}^{(t)}$, and $(\ub{t}_T)_i := \max_{\ub{t}\in T} (\ub{t})_i$. To prove the last line, note that for each $\ub{u}\in \Z_{\geq 0}^{\ell}$ with $(\pi(\ub{u}))_{d+1} = h$, there exists a maximal $T = T_{\ub{u}} \subseteq \mathcal{M}^{(t)}$ such that $\ub{u} \geq \ub{t}_{T}$. The term $\#\{ \ub{u}\in \Z_{\geq 0}^{\ell} ~|~ (\pi(\ub{u}))_{d+1} = h\}$ corresponds to those $\ub{u}$ with $T_{\ub{u}} = \varnothing$. The remaining $\ub{u}$ with $T_{\ub{u}} \neq \varnothing$ (which are exactly the elements of $\mathcal{U}^{(t)}$) are counted $\sum_{U\subseteq T, U\neq \varnothing} (-1)^{|U|} = (1-1)^{k}-1 = -1$ times, yielding the term $\#\{ \ub{u}\in \mathcal{U}^{(t)} ~|~ (\pi(\ub{u}))_{d+1} = h\}$.
 
 Thus, by the standard ``stars and bars'' argument, we obtain from \eqref{incexc} that
 \begin{equation*}
  \#\{ \ub{u}\in \Z_{\geq 0}^{\ell} ~|~ \ub{u} \notin \mathcal{U}^{(t)} \text{ and } (\pi(\ub{u}))_{d+1} = h\} = \sum_{T\subseteq \mathcal{M}^{(t)}} (-1)^{|T|} \binom{h-\sigma(\ub{t}_T) + \ell -1}{\ell - 1},
 \end{equation*}
 where $\sigma(\ub{t}_T) := \sum_i (\ub{t}_T)_i$, provided $h\geq \sigma(\ub{t}_T)$ for every $T\subseteq \mathcal{M}^{(t)}$. Similarly, 
 \begin{align*}
  \#\{ \ub{u}\in \Z_{\geq 0}^{\ell} ~|~ \ub{u} \notin \mathcal{U}^{(t-1)} \text{ and } (\pi(\ub{u}))_{d+1} = h\} &= \sum_{S\subseteq \mathcal{M}^{(t-1)}} (-1)^{|S|} \binom{h-\sigma(\ub{t}_S) + \ell -1}{\ell - 1},
 \end{align*}
 therefore \eqref{hAt} is polynomial in $h$ for $h\geq h^{\mathrm{Kh}}_t(A)$, where
 \[ h^{\mathrm{Kh}}_t(A) \leq \max\Bigg\{\sum_{i} \max_{\ub{t}\in \mathcal{M}^{(t)}} (\ub{t})_i,\ \sum_{i} \max_{\ub{t}\in \mathcal{M}^{(t-1)}} (\ub{t})_i \Bigg\} < \infty \]
 (since $\mathcal{M}^{(t)}$, $\mathcal{M}^{(t-1)}$ are finite), completing the proof.\hfill$\square$

\appendix
\section{\texorpdfstring{$(hA)^{(t)}$}{(hA)\^{}(t)} is always structured if \texorpdfstring{$|A|=3$}{|A|=3}}\label{appA}
 Let $A=\{0 = a_0 < a_1 <\cdots < a_{\ell} < a_{\ell+1} = m\}\subseteq \Z$ be a finite set of integers with $\gcd(A) = 1$, and let $t\geq 1$ be a fixed integer. Nathanson's original definition of structure in \cite{natSOFSI, natSOFSII} is as follows: there exist $c$, $d\in \Z_{\geq 0}$ and subsets $C\subseteq [0,c-2]$, $D\subseteq [0,d-2]$, all depending on $A$ and $t$, such that
 \begin{equation}
  (hA)^{(t)} = C \cup [c,\, hm-d] \cup (hm-D). \label{natdef}
 \end{equation}
 If one requires that $c\leq hm-d$ in order for the interval to be non-empty, then \eqref{natdef} becomes equivalent to having $(hA)^{(t)} = [hm]\setminus \big(\mathcal{E}_t(A) \cup (hm-\mathcal{E}_t(m-A)) \big)$ (as in \eqref{strt}) \emph{and} $h > (\Fr_t(A) + \Fr_t(m-A) + 1)/m$.\footnote{If $\{n+1,\ldots,n+m\}\subseteq \P_t(A)$, then since $m+\P_t(A) \subseteq \P_t(A)$, we have $\{n+1,n+2,\ldots\}\subseteq \P_t(A)$. So if $\{n+1,\ldots,n+m\}\subseteq [hm]\setminus \big(\mathcal{E}_t(A) \cup (hm-\mathcal{E}_t(m-A)) \big)$, then we must have $n\geq \Fr_t(A)$ and $n+m\leq hm-\Fr_t(m-A)-1$.} For instance, Yang--Zhou's \cite{yanzho21} ``$h\geq \sum_{i=2}^{\ell+1} (ta_i - 1) - 1$'' is sharp in the sense that, for every $m\geq 4$ and $t\geq 1$, the set $A=\{0,m-1,m\}$ is such that $(hA)^{(t)}$ satisfies \eqref{natdef} with a non-empty interval only for $h\geq tm-2$. At the same time, in this appendix we will show that if $|A|=3$, then $(hA)^{(t)}$ is structured for every $h$, $t\geq 1$, and $\Fr_t(A) = tam-a-m$. The case $t=1$ is Theorem 4 of Granville--Shakan \cite{grasha20}.
 
 Let $m\geq 2$ be an integer, and $A = \{0<a<m\}$ with $\gcd(a,m) = 1$. The first lemma is Theorem 4.4.1 of Ram\'irez-Alfons\'in \cite{alfonsin06}, written as in Austin \cite{austinMcN}.
 
 \begin{lem}\label{mcnuggs}
  Write $\{x\} := x-\lfloor x\rfloor$. For any $n\in \Z_{\geq 0}$, we have
  \begin{equation*}
   r_A(n) = \frac{n}{am} - \left\{\frac{na^{-1}}{m} \right\} - \left\{\frac{nm^{-1}}{a} \right\} + 1,
  \end{equation*}
  where $a^{-1}$ \textnormal{(resp. $m^{-1}$)} is an inverse of $a\pmod{m}$ \textnormal{(resp. $m\pmod{a}$)}.
 \end{lem}
 \begin{proof}
  Let $u_0$ be the unique $0\leq u_0< m$ such that $\lmod{n}{au_0}{m}$. Then, as $\gcd(a,m)=1$, there is $v_0\in \Z_{>-a}$ for which $n = au_0 + mv_0$. If $v_0\geq 0$, then the set of solutions to $n = au+mv$ ($u,v\in \Z_{\geq 0}$) is described by
  \begin{equation}
   u = u_0 +im,\ v = v_0 - ia \qquad \text{for } 0\leq i\leq \lfloor v_0/a\rfloor,\label{solspr}
  \end{equation}
  while if $v_0<0$, there are no solutions $u,v\in\Z_{\geq 0}$. Therefore, if $r_{A}(n) = t$ for some $t\geq 0$, we have $\lfloor v_0/a\rfloor = t-1$, and thus
  \begin{equation}
   v_0 = (t-1)a + \left\{\frac{v_0}{a}\right\}a. \label{v0eq}
  \end{equation}
  This implies that
  \begin{equation}
   n= au_0 + am(t-1) + \{v_0/a\}am, \label{gv0}
  \end{equation}
  so
  \begin{align*}
   r_{A}(n) = t &= \frac{n}{am} - \frac{u_0}{m} - \left\{\frac{v_0}{a}\right\} + 1 \\
   &= \frac{n}{am} - \left\{\frac{na^{-1}}{m}\right\} - \left\{\frac{nm^{-1}}{a}\right\} + 1.\qedhere
  \end{align*}
 \end{proof}
 
 \begin{cor}\label{EtA}
  We have:
  \begin{enumerate}[label=\textnormal{(\roman*)}]
  \item $\P_t(A) = (t-1)am + \P(A)$;\medskip
  
  \item $\mathcal{E}_t(A) = [(t-1)am-1] \cup \big((t-1)am + \mathcal{E}(A)\big)$.
 \end{enumerate}
 \end{cor}
 \begin{proof}
  By Lemma \ref{mcnuggs}, note that for $N\geq 1$ and $0\leq k<am$, we have $r_{A}(k)\leq 1$ and
  \begin{align*}
   r_A(Nam + k) &= N + \frac{k}{am} - \left\{\frac{ka^{-1}}{m} \right\} - \left\{\frac{km^{-1}}{a} \right\} + 1 \\
   &= N + r_{A}(k).
  \end{align*}
  so $\min_{n\in \P_t(A)} n = (t-1)am$, corresponding to $N=t-1$ and $k=0$. For $n\geq (t-1)am$, it follows that $r_A(n-(t-1)am) \geq 1$ (i.e., $n-(t-1)am \in \mathcal{P}(A)$) if and only if $r_A(n) \geq t$ (i.e., $n\in \mathcal{P}_t(A)$), concluding part (i). For part (ii),
  \begin{align*}
   \mathcal{E}_t(A) = \Z_{\geq 0}\setminus \P_t(A) &= \Z_{\geq 0}\setminus\big((t-1)am +\P(A)\big) \\
   &= [(t-1)am-1] \cup \big((t-1)am + \mathcal{E}(A)\big).\qedhere
  \end{align*}
 \end{proof}
 
 \begin{cor}\label{3prop}
  We have:
  \begin{enumerate}[label=\textnormal{(\roman*)}]
  \item $\Fr_t(A) = tam-a-m$;\medskip
  
  \item $\#\mathcal{E}_t(A) = (t-1)am + \tfrac{1}{2}(a-1)(m-1)$.
 \end{enumerate}
 \end{cor}
 \begin{proof}
  By Corollary \ref{EtA}, it suffices to show the result for $t=1$; that is, $\Fr(A) = am-a-m$ and $\#\mathcal{E}(A) = \frac{1}{2}(a-1)(m-1)$. These are Theorems 2.1.1 and 5.1.1 in Ram\'irez-Alfons\'in \cite{alfonsin06}, respectively, to which we provide a short proof.
  
  First, using Lemma \ref{mcnuggs} one checks that $r_A(am-a-m) =0$, while $r_A(n) > 0$ whenever $n>am-a-m$, thus proving (i). For (ii), we claim that $r_A(n) + r_A(am-n) = 1$ for every $0\leq n \leq am$ such that $a\nmid n$, $m\nmid n$. Indeed,
  \begin{align*}
   r_A(am -n) &= \frac{am-n}{am} - \bigg\{\frac{(am-n)a^{-1}}{m}\bigg\} - \bigg\{\frac{(am-n)m^{-1}}{a} \bigg\} +1 \\
   &= 1 - \frac{n}{am} -\bigg\{-\frac{na^{-1}}{m}\bigg\} -\bigg\{-\frac{nm^{-1}}{a}\bigg\} + 1 \\
   &= 1 - \frac{n}{am} +\bigg\{\frac{na^{-1}}{m}\bigg\} +\bigg\{\frac{nm^{-1}}{a}\bigg\} - 1 = 1- r_A(n).
  \end{align*}
  So, for $n$ such that $a\nmid n$, $m\nmid n$, we have $n\in[am]\setminus\mathcal{E}(A)$ if and only if $am-n\in \mathcal{E}(A)$  This, together with the fact that multiples of $a$ or $m$ are representable, yields
  \begin{align*}
   \#\mathcal{E}(A) &= \#[am] - \#(am - \mathcal{E}(A)) - \#\{n\in [am] ~;~ a\mid n \text{ or } m \mid n \} \\
   &= (am+1) - \#\mathcal{E}(A) - (a+m)  = (a-1)(m-1) - \#\mathcal{E}(A),
  \end{align*}
  therefore $\#\mathcal{E}(A) = \frac{1}{2}(a-1)(m-1)$.
 \end{proof}
 
 \begin{thm}\label{str3}
  For every $h$, $t\geq 1$,
  \[ (hA)^{(t)} =[hm] \setminus\big(\mathcal{E}_t(A) \cup (hm- \mathcal{E}_t(m-A)) \big); \]
  i.e., $(hA)^{(t)}$ is always structured.
 \end{thm}
 \begin{proof}
  Let $h\geq 1$, and take $n\in[hm]$ with $n\notin \mathcal{E}_t(A) \cup (hm - \mathcal{E}_t(m-A))$ -- i.e., $r_A(n)$, $r_{m-A}(hm-n) \geq t$. We want to show that this implies that $n\in (hA)^{(t)}$. As in \eqref{solspr}, since $i$ can go up to at least $t-1$, we have at least $t$ representations $n = au+mv$ with $u$, $v$ satisfying, by \eqref{v0eq},
  \begin{align*}
   u+ v &= u_0 + v_0 +i(m-a) \\
   &\leq u_0 + v_0 + (t-1)m  - (t-1)a \\
   &= u_0 + (t-1)m + \left\{\frac{v_0}{a}\right\}a.
  \end{align*}
  
  Now we apply the same argument to $m-A$. Let $u'_0$ be the unique $0\leq u'_0<m$ such that $\lmod{hm-n}{(m-a)u'_0}{m}$, and $v'_0$ such that $hm-n = (m-a)u'_0 + mv'_0$. But $0 \equiv hm = n + (hm-n) \equiv au_0 + (m-a)u'_0 \equiv a(u_0 - u'_0) \pmod{m}$, implying that $u_0 = u'_0$. Thus, applying \eqref{gv0} to $n$ (with respect to $A$) and to $hm-n$ (with respect to $m-A$), we get
  \begin{align*}
   hm &= n + (hm-n) = \bigg(au_0 + (t-1)am + \bigg\{\frac{v_0}{a}\bigg\}am\bigg)\\
   &\hspace{9em}+ \bigg((m-a)u'_0 + (t-1)(m-a)m + \bigg\{\frac{v'_0}{m-a}\bigg\}(m-a)m\bigg) \\
   &= a(u_0 - u'_0)  + m\bigg( u'_0 + (t-1)m + \bigg\{\frac{v_0}{a}\bigg\}a + \bigg\{\frac{v'_0}{m-a}\bigg\}(m-a) \bigg) \\
   &\geq m\bigg( u_0 + (t-1)m + \bigg\{\frac{v_0}{a}\bigg\}a \bigg),
  \end{align*}
  and so $u+v \leq h$, implying that $r_{A,h}(n) \geq t$.
 \end{proof}
 
\addtocontents{toc}{\protect\setcounter{tocdepth}{0}}
\section*{Acknowledgements}
 I'm grateful to Andrew Granville for his invaluable advice and several insightful suggestions during the writing of this article. I also thank Cihan Sabuncu for the numerous discussions, and the anonymous referees for their very careful reading.
 
\addtocontents{toc}{\protect\setcounter{tocdepth}{1}}

\bibliographystyle{amsplain}
\bibliography{$HOME/Academie/Recherche/_latex/bibliotheca}%
\end{document}